\documentclass[11pt, final]{amsart}

\usepackage{a4wide}
\usepackage{amsmath,amssymb,amsfonts,amsthm}
\usepackage{bm}
\usepackage{mathtools}
\usepackage{xcolor}
\usepackage[marginparwidth=2cm]{geometry} 
\usepackage{float}
\usepackage{xargs}
\usepackage[colorlinks,citecolor=blue,urlcolor=blue,bookmarks=false,hypertexnames=true]{hyperref}
\usepackage{ stmaryrd }
\usepackage{ esint }
\usepackage{mathrsfs}

\usepackage[maxbibnames=99, giveninits=true]{biblatex}
\theoremstyle{plain}

\newtheorem{theorem}{Theorem}[section]

\newcounter{assumpA}
\renewcommand{\theassumpA}{A}
\newenvironment{assumptionA}[1][]%
{%
 ~\refstepcounter{assumpA}%
  \paragraph*{\textbf{Property
  \theassumpA\if\relax\detokenize{#1}\relax\else\ (#1)\fi}}%
}%
{\par}

\newcounter{assumpB}
\renewcommand{\theassumpB}{B}
\newenvironment{assumptionB}[1][]%
{%
 ~\refstepcounter{assumpB}%
  \paragraph*{\textbf{Property
  \theassumpB\if\relax\detokenize{#1}\relax\else\ (#1)\fi}}%
}%
{\par}

\newtheorem{lemma}[theorem]{Lemma}
\newtheorem{proposition}[theorem]{Proposition}

\theoremstyle{definition}
\begingroup
\newtheorem{remark}[theorem]{Remark}

\newtheorem{definition}[theorem]{Definition}
\endgroup

\numberwithin{equation}{section}

\newcommand{\R}{\mathbb{R}}

\newcommand{\Z}{\mathbb{Z}}
\newcommand{\N}{\mathbb{N}}
\renewcommand{\H}{\mathcal{H}}
\def\e{\varepsilon}
\newcommand{\ud}{\,\textnormal{d}}
\def\A{\mathcal A_Q}
\def\bd{\partial}

\def\M{\mathbb M}

\def\G{\mathcal G}

\title[Non-polyconvex $Q$-integrands with lower semicontinuous energies]{Non-polyconvex $Q$-integrands with \\ lower semicontinuous energies}
\author[D. De Gennaro]{Daniele De Gennaro}
\email{daniele.degennaro@unibocconi.it}
\author[A. De Rosa]{Antonio De Rosa$^*$}\thanks{$^*$Antonio De Rosa is the corresponding author.}
\email{antonio.derosa@unibocconi.it}

\address{Department of Decision Sciences and BIDSA, Bocconi University, Milan, Italy}

\date{}

\begin{document}
\begin{abstract}
For any fixed \(Q \in \mathbb{N}\), we construct a positive measure on the space of positively oriented \(2\)-vectors in \(\mathbb{R}^4\), whose barycenter is a simple \(2\)-vector, yet which cannot be approximated by weighted Gaussian images of Lipschitz \(Q\)-graphs. The construction extends to positively oriented \(m\)-vectors in \(\mathbb{R}^n\) whenever \(n-2 \ge m\geq 2\). This geometric obstruction implies that the approximation result established in \cite{DeRLeiYou} is sharp: all \(Q \in \mathbb{N}\) are indeed necessary to ensure the density of weighted Gaussian images of Lipschitz multigraphs in the space of positive measures with simple barycenter.
As an application, we prove that for every $Q\geq 1$ and $p\ge 2$ there exists a non-polyconvex $Q$-integrand  whose associated energy is weakly lower semicontinuous in $W^{1,p}$. 
This also provides new insight into the question posed in \cite[Remark~1.14]{DeRLeiYou}.
\end{abstract}

\maketitle

\section{Introduction}
Given a function \(\psi: \mathbb{R}^{(n-m)\times m} \to [0,\infty)\) and an integer \(Q \in \mathbb{N}\), we define the \(\psi\)-energy of a weakly differentiable \(Q\)-valued map \(f: B \subseteq \mathbb{R}^m \to \mathcal{A}_Q(\mathbb{R}^{n-m})\) as
\begin{equation}\label{functional}
    E_\psi(f) = \int_B \sum_{i=1}^Q \psi\big((Df(x))_i\big)\, dx \;=:\; \int_B \bar\psi(Df(x))\, dx.
\end{equation}
Although this definition depends on the value of \(Q\), we will omit this dependence from the notation for simplicity.
This energy was studied by De Rosa, Lei and Young in \cite{DeRLeiYou}, and the definition is well posed, as it is invariant under permutations of $\{1,\ldots,Q\}$. Geometrically, \(E_\psi(f)\) is relevant as it represents the \(\psi\)-anisotropic energy of the \(Q\)-valued graph of \(f\), interpreted as a current.

The functional \(E_\psi\) is a particular case of a broader class of energies associated with \(Q\)-integrands, introduced by De Lellis, Focardi, and Spadaro in \cite[Section~0.1]{DeLFocSpa}. A \(Q\)-integrand is a function
$F: (\R^{n-m})^Q\times  ( \R^{(n-m)\times m})^Q\to [0,\infty)$ satisfying for every permutation $\sigma$ of $\{1, . . . , Q\}$ the invariance
$$F(a_1,...,a_Q,X_1,...,X_Q) = F(a_{\sigma(1)},...,a_{\sigma(Q)},X_{\sigma(1)},...,X_{\sigma(Q)}).$$ 
Thus, such integrand induces an anisotropic energy functional on \(Q\)-valued maps via
\[
\mathcal F: f\in Lip(B; \A(\R^{n-m}))\mapsto \int_B F\big((f(x))_1,\ldots,(f(x))_Q,(Df(x))_1, \ldots, (Df(x))_Q\big)\, dx.
\]
While the integrand \(\bar\psi\) defined in \eqref{functional} is indeed a \(Q\)-integrand, the converse is not generally true: not every \(Q\)-integrand arises from a scalar function \(\psi\) as in   \eqref{functional}.

In the study of existence and regularity of minimizers \cite{Alm68,Alm76,DLDRG,DPDRG2,DR,Evans,KM,KT,SSA}—and more generally, of critical points—for anisotropic geometric variational problems, it is essential to impose suitable ellipticity conditions on the anisotropic integrands; see, for instance, the survey \cite{DeR-survey}.

For \(Q = 1\), the notions of \emph{polyconvexity} and \emph{quasiconvexity} for \(\psi\) are classical in the nonlinear analysis literature; see \cite[Definition~4.2]{Mull}. Polyconvexity corresponds to convexity in the space of \(d\)-vectors, as clarified in \cite[Remark~1.4]{DeRLeiYou}, and plays a central role in the regularity theory of anisotropic minimal graphs \cite{DRR,DRT,GT,Sz,T21}. Quasiconvexity, on the other hand, is equivalent to the lower semicontinuity of the energy \(E_\psi\) under growth conditions, and is therefore fundamental in establishing the existence of minimizers.  We refer the reader to the foundational works by Acerbi and Fusco \cite{AcerbiFusco1984}, Fonseca \cite{Fons}, and Fonseca and M\"uller \cite{FoMu}, and to the books \cite{Dac-book,Mull,Rindler2018} for details.

It is by now classical that polyconvexity implies quasiconvexity, but the converse fails in general. In particular, Šverák \cite{Sve}, and independently Alibert and Dacorogna \cite{AliDac}, in the \(Q=1\) setting constructed a continuous function \(\psi : \mathbb{R}^{2 \times 2} \to [0,\infty)\) that is quasiconvex but not polyconvex, see also \cite{IK}. Earlier examples of quasiconvex quadratic forms on \(\mathbb{R}^{3 \times 3}\) that are not polyconvex can be found in \cite{Ball, Serre, Terpstra}. However, it is known that every quasiconvex quadratic form on \(\mathbb{R}^{2 \times 2}\) is polyconvex \cite{Serre, Terpstra}, which highlights the significance of the counterexamples in \cite{Sve, AliDac}. In general, constructing explicit quasiconvex integrands is not an easy task, see \cite{Sve2}.

The gap between polyconvexity and quasiconvexity is intimately related to the existence of positive measures on positively oriented \(2\)-vectors in \(\mathbb{R}^4\) whose barycenter is a simple \(2\)-vector, but which cannot be approximated by weighted Gaussian images of Lipschitz graphs. We refer the reader to \cite[Chapters~4,5,6]{Rindler2018} for an extensive discussion on Young measures that are not gradient Young measures, and the connection with quasiconvexity. {As it will become clearer after Section~\ref{sec:prelim},
	polyconvexity and quasiconvexity correspond to Jensen-type inequalities for
	different classes of measures. For a   polyconvex integrand, Jensen inequality holds for every probability measure on the
	positively oriented simple \(2\)-vectors  with simple barycenter.
	Quasiconvexity imposes Jensen inequality only for measures generated by
	gradients, that is, measures that are limits of weighted Gaussian images of Lipschitz
	graphs. Hence a positive measure with simple barycenter which is not generated
	by Lipschitz graphs can be potentially used to show the existence of quasiconvex integrands which are not polyconvex.}

For every \(Q \in \mathbb{N}\), the analogues of polyconvexity and quasiconvexity for \(Q\)-integrands were introduced and studied by De Lellis, Focardi, and Spadaro \cite{DeLFocSpa}, building on earlier work by Mattila \cite{Mat}. The main result of \cite{DeLFocSpa} establishes the equivalence between the quasiconvexity of continuous \(Q\)-integrands and the lower semicontinuity of the associated energy functionals.  As in the classical case, quasiconvexity is thus the appropriate notion to ensure the existence of minimizers among \(Q\)-valued maps. Already in \cite[Introduction]{Mat}, Mattila raised the question of the relationship between these ellipticity conditions for $Q$-integrands.

One of the main results of this paper is to show that, analogously to classical $1$-integrands, the lower semicontinuity of the energy functional associated to a lower semicontinuous $Q$-integrand  does not imply polyconvexity. We state below the theorem, referring the reader to Section~\ref{sec:prelim} for the relevant notation and terminology: 
\begin{theorem}\label{thm:qcvx_non_poly}
    Let $m=2$ and $n=4$.
    For every $Q\ge 1$ and $p\ge 2$ there exists a non-polyconvex  $Q$-integrand  whose associated energy is weakly lower semicontinuous in $W^{1,p}$.
\end{theorem}
We do not claim that the integrand in Theorem~\ref{thm:qcvx_non_poly} is quasiconvex.
As previously mentioned, in the case of \textit{continuous} $Q$-integrands,  quasiconvexity is equivalent to  weak lower semicontinuity of the associated energy \cite[Theorem~0.2]{DeLFocSpa}. However, continuity of the integrand is not guaranteed in our setting, where the integrand is constructed as an envelope \eqref{eq:quasi_envel} that may in general only be lower semicontinuous. 
{Consequently, the quasiconvexity of the resulting envelope, in the sense of Definition~\ref{def:q-qcvx}, is not a consequence of the lower semicontinuity of the energy proved in Theorem \ref{thm:qcvx_non_poly}, and remains unclear in general. We establish instead a slightly different property of the constructed integrand (see   Property~\ref{hp:A}), which is sufficient to prove the lower semicontinuity of the associated energy in Theorem \ref{thm:qcvx_non_poly}.}
For this reason we formulate Theorem~\ref{thm:qcvx_non_poly} directly in terms of lower semicontinuity of the associated energy. In particular, the theorem shows that polyconvexity is not necessary for lower semicontinuity of the energy in the $Q$-valued framework; see also the discussion below and Section~\ref{sec:final_rmks}.

Theorem~\ref{thm:qcvx_non_poly} relies on our main technical contribution, which we state below, that establishes the existence, for any fixed \(Q \in \mathbb{N}\), of a measure on positively oriented \(2\)-vectors in \(\mathbb{R}^4\), with simple barycenter, that cannot be approximated by weighted Gaussian images of Lipschitz \(Q\)-graphs. We again refer the reader to Section~\ref{sec:prelim} for the related notation:
\begin{theorem}\label{thm:notapprox}
    For any fixed $Q\in\N$, there exists a measure $\mu$ on the positively oriented Grassmannian $\widetilde{Gr}^+(2,4)$ whose barycenter is parallel to the simple 2-vector $e_1\wedge e_2$, with the following property.  The measure $\mu$  cannot be approximated (in the weak-* topology) by weighted Gaussian images of $Q$-graphs associated to $Q$-valued Lipschitz functions $f:B \subseteq \R^2\to \mathcal A_Q(\R^2)$ with $f\llcorner {\bd B }\equiv Q[0]$.
\end{theorem}

Theorem~\ref{thm:notapprox} should be compared with \cite[Theorem~1.5]{DeRLeiYou}, where De Rosa, Lei, and Young proved that every measure supported on positively oriented \(m\)-vectors in \(\mathbb{R}^n\), whose barycenter is a simple \(m\)-vector, can be approximated by weighted Gaussian images of Lipschitz {$Q$-valued graphs,  with $Q$ possibly tending to $+\infty$}. Our result thus shows that \cite[Theorem~1.5]{DeRLeiYou} is sharp: {in some cases it is necessary to allow $Q$ to tend to $+\infty$ for the approximation to hold.}

The applications of \cite[Theorem~1.5]{DeRLeiYou} were twofold. First, it was crucial in the proof of \cite[Theorem 1.16]{DeRLeiYou}, establishing that polyconvexity of geometric integrands is necessary for the validity of the rectifiability theorem for anisotropic stationary varifolds proved by De Philippis, De Rosa, and Ghiraldin \cite[Theorem~1.2]{DPDRG}. This strengthened the earlier result \cite[Theorem~A]{DRK}, which showed the necessity of quasiconvexity for rectifiability. Second, \cite[Theorem~1.5]{DeRLeiYou} played a central role in proving that a classical integrand \(\psi\) is elliptic for multigraphs—i.e., the function \(\bar{\psi}\) defined in \eqref{functional} is a quasiconvex \(Q\)-integrand for every \(Q \in \mathbb{N}\)—if and only if \(\psi\) is polyconvex; see \cite[Theorem~1.8]{DeRLeiYou}.

{We stress that \cite[Remark~1.14]{DeRLeiYou} remains an open question: whether there exists $Q \in \mathbb N$ such that, for every continuous
integrand
$\psi \in C^0(\mathbb R^{(n-m)\times m},[0,\infty))$, the condition that
$\bar\psi$ defined in \eqref{functional} is a quasiconvex $Q$-integrand  implies that $\psi$ is
polyconvex. As discussed above, this question is closely tied to the approximability of a measure on the oriented Grassmannian with simple barycenter by means of Gaussian images of Lipschitz $Q$-valued maps with $Q$ bounded.}

{At the level
of measures, Theorem~\ref{thm:notapprox} shows that such an
approximation with any fixed $Q$ fails: for every fixed $Q$ we
exhibit a positive measure with simple barycenter which is not a weak-*
limit of weighted Gaussian images of Lipschitz $Q$-graphs. Therefore,
for any fixed $Q$, $Q$-quasiconvexity does not provide all the Jensen inequalities needed for  polyconvexity. This gives a
geometric reason why the use of arbitrarily large $Q$ in
\cite[Theorem~1.8]{DeRLeiYou} is essential. The remaining difficulty is analytic: one
would need a precise characterization, within the restricted class of
integrands of the form $\bar\psi$, of the relevant quasiconvex envelope.
The natural candidate envelope discussed in Section~4 may be merely lower
semicontinuous, and hence may lie outside the continuous framework of
\cite{DeLFocSpa}.  A further obstruction is specific to the multivalued setting. The natural
envelope is defined on pairs $(a,X)\in \mathcal A_Q(\mathbb R^{n-m}\times
\mathbb R^{(n-m)\times m})$, and it is not clear whether it can be represented
by an integrand depending only on the gradient variable, let alone by an
integrand of the form $\bar\psi$. This reflects the fact that, for
$Q$-valued maps, the differential alone does not determine the first-order
Taylor approximation. These obstructions are discussed further in
Remark~\ref{rmk:anisotropic}.}

\vspace{2ex}

{\bf Structure of the paper.} 
In Section~\ref{sec:prelim}, we collect the necessary preliminaries and fix the notation. 

In Section~\ref{sec:nonapprox}, we prove Theorem~\ref{thm:notapprox}, the first main result of the paper. More precisely, we construct a measure \(\mu\) on the positively oriented Grassmannian \(\widetilde{Gr}^+(2,4)\), whose barycenter is a simple \(2\)-vector, and which cannot be approximated--for fixed \(Q \in \mathbb{N}\)--by weighted Gaussian images of Lipschitz \(Q\)-graphs. While our construction is carried out in dimensions \((2,4)\), it can be extended to \(\widetilde{Gr}^+(m,n)\) for every \(n-2 \geq m\geq 2\), as outlined in Remark~\ref{rmk:proj-2}.

This result is the multivalued graphical analogue of a theorem by Burago and Ivanov \cite[Theorem~4]{BurIva}, which demonstrates the existence of measures on the oriented Grassmannian \(\widetilde{Gr}(2,4)\) that cannot be approximated by weighted Gaussian images of surfaces bounded by simple planar curves. In particular, Theorem~\ref{thm:notapprox} shows that the approximation theorem \cite[Theorem~1.5]{DeRLeiYou} is sharp: although \cite{DeRLeiYou} establishes that every positive measure on \(\widetilde{Gr}^+(m,n)\) with simple barycenter can be approximated by Gaussian images of Lipschitz multigraphs, this approximation requires allowing arbitrary values of \(Q\).
Beyond their intrinsic geometric significance, such filling results have key applications to the study of ellipticity conditions for anisotropic integrands. Indeed, \cite[Theorem~1.8]{DeRLeiYou} depends crucially on \cite[Theorem~1.5]{DeRLeiYou}, just as the equivalence between elliptic and convex hulls in \cite[Theorem~3]{BurIva} implies the approximation result \cite[Theorem~2]{BurIva}.

In Section~\ref{sec:quasi/poly} we construct a non-polyconvex \(Q\)-integrand with associated energy that is lower semicontinuous. To this aim, we first define a $1$-integrand $\psi$ on $\R^{2\times 2}$, which vanishes on (and only on) the support of the measure $\mu$ constructed in Theorem~\ref{thm:notapprox}. Then we would like to characterize the quasiconvex envelope of the associated $Q$-integrand $\bar \psi$. 
However, carrying this out is technically challenging. The key difficulty is that the usual explicit candidate $A$ for the quasiconvex envelope \eqref{eq:quasi_envel}—in the spirit of Dacorogna’s formula for quasiconvex envelopes of $1$-integrands (see \cite[Section~6.3]{Dac-book}, \cite{Dac})—produces a $Q$-integrand that in general may be just lower semicontinuous, rather than continuous.
As discussed in greater detail in Remark \ref{rmk:A_not_C0}, if a sequence of $Q$-matrices with $Q$ distinct values of multiplicity one converges to a $Q$-matrix supported on the same value with multiplicity $Q$, there may be a lower semicontinuity gap. Indeed  the class of competitors for the envelope \eqref{eq:quasi_envel} evaluated at the limit is strictly larger than the class of competitors for the envelope evaluated along the sequence.
As a result, the continuous framework developed in \cite{DeLFocSpa} cannot be directly applied. We thus  undertake a careful analysis of the envelope $A$ of the  integrand $\psi$, culminating in Theorem~\ref{prop:quasi_envel}. In this result, we show that the integrand $A$ gives rise to a weakly lower semicontinuous energy on $W^{1,p}$ for every $p\ge 2.$

The potential loss of continuity motivates us to introduce a slight variation of quasiconvexity: Property~\ref{hp:A}. For lower semicontinuous integrands with the same type of dependence as $A$, this condition is essentially equivalent to weak lower semicontinuity of the associated energy functional, as shown in Proposition~\ref{prop:iff}. In general, this condition is weaker than the quasiconvexity introduced in \cite{DeLFocSpa}. Nonetheless, when the $Q$-integrand is continuous, Property~\ref{hp:A} is equivalent to  quasiconvexity.

A key technical tool we use in the proofs of Section~\ref{sec:quasi/poly} is the piecewise affine approximation Lemma~\ref{lem:polyapprox}, proved in Section~\ref{sec:piecewise}. Inspired by \cite{DanPra}, it provides an approximation of the energy of a Lipschitz $Q$-valued map by $Q$-valued functions that are piecewise affine outside a set of arbitrarily small measure. This lemma is instrumental in our study of quasiconvex envelopes for $Q$-integrands and we believe it is of independent interest.

\vspace{2ex}

{\bf Acknowledgment.} Both authors were funded by the European Union: the European Research Council (ERC), through StG ``ANGEVA'', project number: 101076411. Views and opinions expressed are however those of the authors only and do not necessarily reflect those of the European Union or the European Research Council. Neither the European Union nor the granting authority can be held responsible for them.

\section{Notation and preliminaries}\label{sec:prelim}
Let $n,m\in\N$ with $n>m$. Throughout the paper we will mostly focus on $m=2$, $n=4$.

We will  use $B$ to denote the unit open ball in $\R^m$, while $B(x,r)$ denotes the open ball of center $x$ and radius $r$ in $\R^n$. We let $D$ denote the unit cube $(-\frac12,\frac12)^m$ in $\R^m$ and  $D(x,r)=x+(-\tfrac r2,\tfrac r2)^m$. We will also denote by $\{e_i\}_{i=1,\dots,n}$ the standard basis in $\R^n$, and $e_{ij}=e_i\wedge e_j.$

Furthermore, for every $x\in\R^m$ we denote by $[x]$ the Dirac mass in $x$.

\subsection{Currents}\label{sub:current}

We recall some basic definitions of the theory of currents, referring to \cite{Sim-book} for additional material. 

 We denote by ${\bigwedge^m\R^n}$ the set of $m$-vectors in $\R^n$ and with  $\widetilde{ Gr}(m,n)$ the oriented Grassmannian of oriented $m$-planes in $\R^n$. {We also use ${(\bigwedge^m\R^n)^+}$ to denote the subset of $\bigwedge^m\R^n$ of $m$-vectors whose projection on $\R^m\subset\R^n$ is positively oriented.}
 An $m$-current $T$ is an element in the dual of the space $C^\infty_c(\R^n; {\bigwedge^m\R^n})=:\mathscr D^m(\R^n)$ of  differential $m$-forms on $\R^n$.  The boundary of a current $T$ is the $(m-1)$-current defined by $\bd T(\eta):= T(d\eta )$ for every $\eta \in \mathscr D^{m-1}(\R^n).$ The mass of the $m$-current $T$ is defined as $\M(T):=\sup \{ T(\eta): \eta \in\mathscr D^m(\R^n), |\eta|\le 1 \}$. We define a  rectifiable $m$-current $T\in \textbf{R}^m(\R^n)$ as a couple $T=(M,\theta{\omega} )$, where $M$ is an $m$-rectifiable set in $\R^n$, $\theta:M\to \R$ is a locally $\H^m$-integrable positive function  (the multiplicity) and ${\omega}: M\to {\bigwedge^m\R^n}$ is an $\H^m$-measurable  function (the orientation) such that for $\H^m$-a.e. point $x\in M $, ${\omega}(x)$ can be expressed in the form $v_1\wedge\dots \wedge v_m$, where $v_1,\dots, v_m$ form an orthonormal basis for $T_x M\in\widetilde{Gr}(m,n)$, the approximate tangent space  to $M$ at $x$. 
The orientation field is ${TM}:M\to \widetilde{Gr}(m,n)$.
We will often identify $\widetilde{Gr}(m,n)$ with unit simple $m$-vectors in ${\bigwedge^m\R^n}$ by identifying ${T_xM}$ with ${\omega(x)}$.
We define an  integer-rectifiable $m$-current (integral current) $T\in \textbf{I}^m(\R^n)$ as a rectifiable $m$-current $T=(M,\theta{\omega})$ where $\theta $ takes only  integer values. As a particular instance of integral current,  for every oriented $m$-dimensional rectifiable surface $\Sigma$ in $\R^n$ we denote the fundamental  class of $\Sigma$, i.e., the multiplicity-one integral current associated to $\Sigma$, with $\llbracket \Sigma \rrbracket:=(\Sigma, \omega)\in \textbf{I}^m(\R^n)$, where ${\omega(x)} $  is an $m$-vector representing the oriented tangent space ${{ T_x\Sigma}}$ to $\Sigma$ at $x$.

We also recall the notions of push-forward  and slicing of a current. Given an integer-rectifiable $m$-current $T=(M,\theta{\omega})$  and a Lipschitz function $f:\R^n \to \R$, we define the slice of $T$ by $f$ at $\mathcal H^1$-a.e. $t\in\R$ by $\langle T, f ,t   \rangle := (M_t, \theta_t \omega_t) $ where $M_t=f^{-1}(t)$, $\theta_t=(\theta\llcorner M_t)\chi_{\{\nabla ^Mf\neq 0\}}$ and $\omega_t= {\omega}\llcorner\frac{\nabla^M f}{|\nabla^M f|}$.  For $\mathcal H^1$-a.e. $t\in \R$ the slice $\langle T,f,t\rangle $ is well-defined and is an integral $(m-1)$-current.
Given a Lipschitz map $f:U\subseteq \R^m\to V\subseteq \R^n$  with $V$ open, if $f$ is proper, i.e. it  satisfies $f^{-1}(K)\cap \text{spt}(T)$ is a compact subset of $U$ whenever $K$ is a compact subset of $V$, the push-forward of $T\in \textbf{R}^\ell(\R^m)$ by $f$  is the rectifiable $\ell$-current on $\R^n$ defined by $f_\sharp T(\eta) = \int_M \theta(x)\langle (\eta(f(x)), (df_\sharp {\omega})(x) \rangle \,  d\H^\ell(x)$ for every $\eta \in \mathscr D^\ell(V)$, where  $df_\sharp $ denotes the standard push-forward of $\ell$-vectors induced by the 1-form $df$.

Lastly, for a rectifiable $m$-current $T=(M, \theta {\omega})\in \mathbf R^m(\R^n)$, we define the weighted Gaussian image $\gamma_T\in \mathcal M(\widetilde{Gr}(m,n))$ of $T$ as the pushforward $P_\sharp(\theta \H^m \llcorner M)$ of the measure $\theta \H^m \llcorner M$ under the orientation field $P : M  \to \widetilde{Gr}(m,n)$.

\subsection{Multivalued maps}\label{sub:Q-val maps}

For a given integer $Q\ge1 $ we define the space of $Q$-points in $\mathbb{R}^n$ as 
\begin{equation*}
\A(\mathbb{R}^{n}) :=\left\{\sum_{i=1}^Q[x_i]\,:\,x_i\in\mathbb R^{{n}}\;\textrm{for every  }i=1,\ldots,Q\right\},
\end{equation*}
equipped with the following distance $\G$. Given two $Q$-points $x=\sum_i[x_i],y=\sum_i[y_i]\in\A(\R^n)$ we  denote 
\[
\G(x,y)=\min_{\sigma\in \mathcal P_Q} \sqrt{\sum_i|x_i-y_{\sigma(i)}|^2},
\]
where $\mathcal P_Q$ denotes the group of permutations of $\{1,\dots,Q\}.$
We can identify $\mathcal{A}_Q(\R^{n})$ with the configuration space $(\R^{{n}})^Q/\text{Sym}(Q)$. Whenever clear from the context, we simply write $\A$ instead of $\A(\R^n)$.  Translations are well-defined in $\A$ only by points of full multiplicity. Given $x\in \R^n$ we denote by $\tau_{x}y=\sum_{j}[y_j+x]$ the translation by the point $Q[x]\in\A$. 

{Following the intrinsic approach of \cite{DeLSpa11}, a  $Q$-valued function is a map from an open set $U\subseteq  \R^m$ to the metric space $(\mathcal{A}_Q,\G)$. Continuous, Lipschitz, H{\"o}lder and Lebesgue measurable functions from $U$ to $\A$ are defined in the usual way.
A multigraph  $T_f$  is the graph of the $Q$-valued function $f$,  defined by
$$T_f = \{(x,y)\in U\times \R^{n}, y \in \text{spt}(f(x))\};$$
note that for all $x$, $\text{spt} (f(x))$ has cardinality at most $Q$. }

{Given $f:U\to \A$ and $x\in U$, we say that $f$ is differentiable at $x_0\in U$ if  there exist $Q$ matrices $L_i$
 such that  the first order Taylor approximation 
\[
	T_{x_0} f(x) = \sum_i[L_i\cdot(x-x_0)+f_i(x_0)]
\]
satisfies $\G(f(x),T_{x_0}f(x))=o(|x-x_0|)$ and $L_i=L_j$ if $f_i(x_0)=f_j(x_0)$. The point $\sum_i [L_i]$ is called the differential of $f$ at $x_0$ and will be denoted $Df(x_0)=\sum_i [D_i f(x_0)]$ for $L_i=Df_i(x_0)$.
Let us make the relevant observation that  the differential $Df (x_0)$ of a $Q$-valued function $f$ does not uniquely determine its first-order approximation $T_{x_0} f$.  To overcome this ambiguity,  the differentials $Df_i(x_0)$ need to be associated to the values   $f_i(x_0)$. This will be the reason of introducing  Definition \ref{def:max-dec}.}

{A selection of a $Q$-valued map $f:U\to \A$ on a set $U'\subseteq U\subseteq \R^m$ is a family of single-valued maps $f_1,\dots,f_Q:U'\to\R^n$ such that $f(x)=\sum_{i=1}^Q [f_i(x)]$ for every $x\in U'$. We say the selection is local if $|U\setminus U'|>0$, while  it is global if $|U\setminus U'|=0$. Measurable $Q$-valued functions always admit a global measurable selection. In general, this is not the case if one requires more regularity on the function and on the selection.}

{A measurable $Q$-valued function $f:U\to \A$ is in the Sobolev class
	$W^{1,p}$ ($1\leq p\leq\infty$) if there exists $\varphi\in L^p(U;[0,+\infty))$ such that 
	\begin{itemize}
		\item[(i)] $x\mapsto\G (f(x),T)\in W^{1,p}(U)$ for all $T\in \A$;
		\item[(ii)] $|D\, \G (f, T)|\leq\varphi$ a.e. in $U$ for all $T\in \A$.
	\end{itemize}
As for classical Sobolev maps, an important feature of Sobolev $Q$-valued functions is the existence of the approximate differential (in the sense above) $Df$ almost everywhere.
Moreover, the map $Df$ is $L^p$ integrable, meaning that $\sqrt{\sum_i |Df_i|^2}\in L^p(U)$. A sequence 
$f_k \in W^{1,p}(U,\A)$ converges weakly 
to $f\in W^{1,p}(U,\A)$ for $k \to \infty$, (and we write $f_k\rightharpoonup f$) in $W^{1,p}(U,\A)$,
if  $\int\G(f_k,f)^p\to0$, for $k\to\infty$ and   $\sup_{k}\int |Df_k|^p<\infty$. There is an alternative and  extrinsic way of defining Sobolev $Q$-valued maps as originally proposed by Almgren \cite{Alm00}. This is done by means of Almgren's bi-Lipschitz embedding. Almgren's bi-Lipschitz embedding is an injective and Lipschitz  map  $\bm{\xi}: \A\to \R^N$ for some $N=N(n,Q)$, such that $Lip(\bm{\xi})\le 1$ and $Lip(\bm{\xi}^{-1}\llcorner\bm{\xi}(\A))\le C(n,Q)$. Moreover, there exists a retraction map $\bm{\rho}:\R^N\to \bm{\xi}(\A)$ which is the identity on $\bm{\xi}(\A)$. We refer to \cite{Alm00} and \cite[Theorem 2.1]{DeLSpa11} for the proof of existence of such $\bm\xi$. Then, $f\in W^{1,p}(U,\A)$ if and only if $\bm\xi\circ f\in W^{1,p}(U,\R^N).$ Thanks to this embedding, compactness, traces and other properties of a Sobolev function $f$ can be expressed in terms of $\bm\xi \circ f$. For instance, the trace of \(f\in W^{1,p}(U,\mathcal A_Q)\) is defined through Almgren's embedding: \(f=g\) on \(\partial U\) in the trace sense if
\(\bm\xi\circ f=\bm\xi\circ g\) on \(\partial U\) in the classical Sobolev trace sense.
 }

{We also define a notion of affine $Q$-valued maps. An affine \(Q\)-valued  map is a map of the form
\[
u(x)=\sum_{j=1}^J q_j\llbracket a_j+A_jx\rrbracket,
\qquad \sum_{j=1}^J q_j=Q,
\]
with \(a_i\neq a_j\) for \(i\neq j\). Its differential is $Du=\sum_{j=1}^J q_j\llbracket A_j\rrbracket $.}

A Lipschitz multigraph   naturally supports an integral current \cite[Definition 1.10]{DeLSpa15}.
This can be defined using \cite[Lemma 1.1]{DeLSpa15}, which ensures that, for every Lipschitz $Q$-valued function $f:U\to\A(\R^{n-m})$ defined on a measurable subset $U$ of $ \R^m$, there exists   a countable measurable partition $(U_i)_{i\in\N}\subseteq U$ and {a local selection of} Lipschitz functions $f^i_j: U_i\to \R^{n-m}$ for $j=1,\dots,  Q,$ such that $f\llcorner U_i= \sum_{j=1}^Q [f^i_j]. $ Assuming that each $f^i_j$ is proper, one can define the integral current $\textbf{T}_f$ associated to the multigraph $T_f$ as
$$
\textbf{T}_f:= (\cdot,f(\cdot))_\sharp \llbracket  U \rrbracket := \sum_{i,j}(\cdot,f^i_j(\cdot))_\sharp\llbracket U_i \rrbracket.
$$

We further refer to \cite{DeLSpa11,DeLSpa15,DeLFocSpa} for the relevant results concerning $Q$-valued functions.

\subsection{Anisotropic integrands}
We now recall some notions concerning  anisotropic energies defined on multigraphs, following \cite{DeLFocSpa}.
\begin{definition}\label{def:Q-integr}
    A measurable map $F: (\R^{n-m})^Q\times  ( \R^{(n-m)\times m})^Q\to \R$ is a $Q$-integrand if for every permutation $\sigma\in \mathcal P_Q$ it holds
$$F(a_1,...,a_Q,X_1,...,X_Q) = F(a_{\sigma(1)},...,a_{\sigma(Q)},X_{\sigma(1)},...,X_{\sigma(Q)}).$$  
\end{definition}
In particular, we can evaluate a $Q$-integrand as above  on elements   $(a,X)\in  \A(\R^{n-m}\times\R^{(n-m)\times m})$, 
meaning that $F(a,X)$ is the value of $F$ at a selection  representing $(a,X)$ in $(\R^{n-m}\times\R^{(n-m)\times m})^Q$. 
With a slight abuse of notation in the following we will not distinguish between elements in $(\R^{n-m}\times\R^{(n-m)\times m})^Q$ and $\A(\R^{n-m}\times\R^{(n-m)\times m})$ when evaluating $Q$-integrands.
In particular, given a weakly differentiable $Q$-valued map  $f:\R^{m}\to \A(\R^{n-m})$, the expression $F(f,\nabla f)$ is well-defined almost everywhere. We then define the notion of quasiconvexity and polyconvexity considered in \cite{DeLFocSpa}.

\begin{definition}\label{def:q-qcvx}
    A locally bounded  $Q$-integrand $ F:(\R^{n-m})^Q\times (\R^{(n-m)\times m})^Q\to \R$ is quasiconvex if for every
    $Q$-valued affine map $u(x)=\sum_{j=1}^Jq_j[a_j + X_j\cdot x]$ with  $a_j\neq a_i$ for $i\neq j$ {and $\sum_{j=1}^Jq_j=Q$}, it holds     \begin{equation}\label{eq:quasiconvexity}
		F (u(0),\nabla u(0))\le \fint_{D} F(\underbrace{a_1,\dots,a_1}_{q_1},\dots,\underbrace{a_J,\dots,a_J}_{q_J}, \nabla f_1(x),\dots,\nabla f_J(x))dx 
,  
	\end{equation}
for every collection of maps $f_j\in Lip(D,\mathcal{A}_{q_j}(\R^{n-m}))$ satisfying $f_j\llcorner \bd D= q_j[a_j+X_j\,\cdot]\llcorner\bd D$.
\end{definition} 
Here and in the following, with a slight abuse of notation, with $f\llcorner \bd D = [a +X\,\cdot] \llcorner\bd D$ for $a\in\R^{n-m}$ and $X\in\R^{(n-m)\times m}$  we mean that for every $x\in\bd D$ it holds $f(x)= [a +X \cdot x]$.

\begin{definition}\label{def:poly}
  Let $\tau(n-m,m):=\sum_{j=1}^{\min\{ n-m,m \}} \left(\begin{smallmatrix}
    n-m\\j
\end{smallmatrix}\right) \left(\begin{smallmatrix}
    m \\ j 
\end{smallmatrix}\right) $ and define $ad:\R^{(n-m)\times m}\to \R^{\tau(n-m,m)}$  as $ad(X)=(X,\text{adj}_2(X),\dots,\text{adj}_{\min\{n-m,m\}}(X))$, where $\text{adj}_k X$ stands for the matrix of all $k\times k$ minors of $X$. 
A $Q$-integrand $F: (\R^{n-m})^Q\times  (\R^{(n-m)\times m})^Q \to \R$ is said to be polyconvex if for every $a\in(\R^{n-m})^Q$  there exists a convex function $g_a:(\R^{\tau(n-m,m)})^Q\to \R$ such that 
  \[
  F(a,X_1,\dots, X_Q)= g_a(ad(X_1),\dots, ad(X_Q)).
  \]
\end{definition}

{
\begin{remark}\label{remarkconv}
The definition of polyconvexity in Definition~\ref{def:poly} can equivalently be expressed in terms of simple \(m\)-vectors. Given
\(X\in\mathbb R^{(n-m)\times m}\), set
\[
M(X):=\begin{pmatrix}
	id_{m\times m}\\
	X
\end{pmatrix}\in \R^{n\times m}.
\]
We then define   \(\bigwedge M(X)\) as the wedge product of the columns of \(M(X)\). Then
\(\bigwedge M(X)\in({\bigwedge^m\mathbb R^n})^+\), its \(e_1\wedge\cdots\wedge e_m\)-component is equal to \(1\), and its remaining coordinates, up to the signs fixed by the chosen basis of \({\bigwedge^m\mathbb R^n}\), are precisely the minors of \(X\).
For elements in $(\R^{(n-m)\times m})^Q$,  this construction is applied componentwise:
\[
(X_1,\ldots,X_Q)\mapsto(\bigwedge M(X_1),\dots,\bigwedge M(X_Q)).
\]
Hence, after fixing the order and sign of the coordinates, there exists a linear projection $
P:{\bigwedge^m \R^n}\to \R^{\tau(n-m,m)}, $ which satisfies $   
P(\bigwedge M (X))=ad(X)$. Moreover, considering the affine subspace $\mathcal K=\{ v\in (\bigwedge^m \R^n)^+: v\cdot e_1\wedge\dots\wedge e_m=1 \}$, $P\llcorner\mathcal K$ is a bijection whose  inverse is affine and   assigns to
\(x\in\mathbb R^{\tau(n-m,m)}\) the unique \(m\)-vector in \(\mathcal K\)
whose remaining coordinates are \(x\).
Thus, requiring the existence of  a  convex function $g_a:(\R^{\tau(n-m,m)})^Q\to \R$ such that 
\[
F(a,X_1,\ldots,X_Q)
=
g_a(ad(X_1),\ldots,ad(X_Q)),
\]
is equivalent to require the existence of a convex function $\widetilde g_a$ defined on $(( \bigwedge^m \R^n)^+)^Q$ such that 
\[
F(a,X_1,\ldots,X_Q)
= \widetilde{g}_a(\bigwedge M(X_1),\dots, \bigwedge M(X_Q)).
\]
Indeed, given $g_a$ one could define $\widetilde{g}_a$ by $\widetilde{g}_a(v_1,\dots, v_Q):=g_a(Pv_1,\dots,Pv_Q)$, while given $\widetilde{g}_a$ one can define $g_a$ in terms of the inverse of $P\llcorner\mathcal K$ as $g_a(x_1,\dots x_Q)=\widetilde{g}_a((P\llcorner\mathcal K)^{-1}x_1,\dots, (P\llcorner\mathcal K)^{-1}x_Q)$. The two expressions for $F$ then coincide since 
\[\widetilde g_a(\bigwedge M(X_1),\dots, \bigwedge M(X_Q))
=
g_a(P \bigwedge M(X_1),\dots, P\bigwedge M(X_Q)).
\]
Lastly, 
since   \(P\) is linear and $(P\llcorner\mathcal K)^{-1}$ is affine, one can check that the convexity of $\widetilde{g}_a$ is equivalent to the convexity of   $g_a$.
\end{remark}
}

The notion of quasiconvex  $Q$-integrand is particularly relevant, as in the case of continuous integrands satisfying growth conditions it is equivalent to the lower semicontinuity of the associated energy. Indeed, by \cite[Theorem 0.2]{DeLFocSpa}  the following characterization holds. 

\begin{theorem}[\cite{DeLFocSpa}]\label{thm:q-qcvx_implies_lsc}
    Let $p\in [1,+\infty), U\subseteq \R^m$ be a bounded open set and $F:(\R^{n-m})^Q\times (\R^{(n-m)\times m})^Q\to \R$ be a continuous $Q$-integrand. If $F$ is quasiconvex and satisfies   for some $C>1$:
    \[
    0\le F(a,X)\le C(1+|a|^q+|X|^p)\qquad \quad \forall (a,X)\in (\R^{n-m})^Q\times(\R^{(n-m)\times m})^Q,    
    \]
    where $q=0$ if $p>m$, $q=p^*$ if $p<m$ and $q\ge 1$ is any exponent if $p=m$,
    then the associated energy   
    \[
    \mathcal F: f\in Lip(U, \A(\R^{n-m}))\mapsto \int_U F(f,\nabla f) 
    \]
    is weakly  lower semicontinuous in $W^{1,p}(U,\A(\R^{n-m}))$. Conversely, if $\mathcal F$ is weakly$^*$  lower semicontinuous in $W^{1,\infty}(U,\A(\R^{n-m}))$, then $F$ is quasiconvex.
\end{theorem}
As in the single-valued case, showing that an integrand is quasiconvex is not easy in general, while polyconvexity may be easier to check. By \cite[Theorem 0.5]{DeLFocSpa},  we have the following implication.

\begin{theorem}[\cite{DeLFocSpa}]\label{thm:dlsp}
    Every locally bounded polyconvex $Q$-integrand is quasiconvex.
\end{theorem}

{Theorem~\ref{thm:qcvx_non_poly} is one of the main results of this paper and, by Theorem~\ref{thm:q-qcvx_implies_lsc},  provides an indication that the reverse implication in Theorem~\ref{thm:dlsp} does not hold. However, it is not a counterexample to the reverse implication in Theorem~\ref{thm:dlsp} since, as noted in the introduction, the integrand we construct in Theorem~\ref{thm:qcvx_non_poly} is only semicontinuous: it is not proved to be either quasiconvex or continuous. Hence the theory of \cite{DeLFocSpa} does not apply. In this respect, the desired counterexample should be a $Q$-integrand analogue of the classical results for $1$-integrands due to \v{S}ver\'ak \cite{Sve} and Alibert--Dacorogna \cite{AliDac}, who exhibited a continuous $1$-integrand $\psi:\R^{2\times 2}\to \R$ that is quasiconvex, for graphs of single-valued Lipschitz functions, but not polyconvex.}

\vspace{2ex}

We now compare Definition~\ref{def:q-qcvx} and Definition~\ref{def:poly} with \cite[Definition 1.7]{DeRLeiYou} and \cite[Definition 1.9]{DeRLeiYou} respectively. The latter are tailored to the study of anisotropic integrands on rectifiable currents. 
 Given a $1$-integrand $\psi:\R^{(n-m)\times m}\to \R$, we denote 
\[
\bar \psi (X)=\sum_{j=1}^Q \psi (X_j), \quad \text{ for all }   X_1,\dots,X_Q\in \R^{(n-m)\times m},
\]
which can be extended in the trivial way to $(\R^{n-m})^Q\times (\R^{(n-m)\times m})^Q$.
It is easy to see that this defines a $Q$-integrand.

If $\Psi:{\bigwedge^m\R^n}\to (0,+\infty)$ is a 1-homogeneous function,  an energy $E_\Psi  $ on rectifiable  currents can be defined as follows: for every $T=(E,\theta {\omega})\in \textbf{R}^m(\R^n)$   set 
\[
    E_\Psi(T)=\int_E \theta(x) \Psi({\omega}(x))\, d\H^m(x).
\]
Thanks to the operator $\bigwedge M$ previously introduced, the restriction of $\Psi$ to the  positively oriented  unit simple $m$-vectors can be seen as  a $1$-integrand on $\R^{(n-m)\times m}$. If we let $\psi: \R^{(n-m)\times m}\to (0,\infty)$ with  $\psi(X) := \Psi(\bigwedge M(X))$
for all $X$, by the area formula for $Q$-valued graphs \cite[Corollary 1.11]{DeLSpa15} we have 
$$
\int_U \bar \psi(\nabla f) = E_{\Psi}(\textbf{T}_f)
$$
for any Lipschitz $Q$-valued function $f: U\subseteq \R^m\to \mathcal{A}_Q(\mathbb{R}^{n-m})$. 
It is therefore interesting to study the quasiconvexity of $\bar \psi$ as it relates to the lower semicontinuity of the functional $E_\Psi$ defined on $Q$-valued graphs. 
This is indeed the focus of \cite{DeRLeiYou}. 
Hence, the notion of polyconvexity considered in \cite{DeRLeiYou} refers to the classical polyconvexity \cite[Definition 4.2]{Mull} of $\psi$, requiring the existence of a 1-homogeneous convex function $\Psi: ({\bigwedge}^m \R^n)^+\to \R$ such that $\psi=\Psi\circ{\bigwedge M}.$

Surprisingly, in \cite[Theorem 1.8]{DeRLeiYou} it is proved the following:
\begin{theorem}[\cite{DeRLeiYou}]\label{thm:DLY}
    Let $\psi\in C^0(\R^{(n-m)\times m},[0,+\infty))$. Then $\psi$ is polyconvex if and only if $\bar \psi $ is a  quasiconvex $Q$-integrand  for every $Q\in\N$. 
\end{theorem}
Intuitively, our Theorem~\ref{thm:qcvx_non_poly} indicates that Theorem~\ref{thm:DLY} is a peculiarity of ``geometric'' $Q$-integrands, and not extendable to general $Q$-integrands. One could hope to use our non-approximability Theorem~\ref{thm:notapprox} to show that Theorem~\ref{thm:DLY} fails if the value $Q$ is fixed, however there is a fundamental obstruction in doing so, that we discuss in Remark~\ref{rmk:anisotropic}.

\section{A non-approximable Grassmannian measure: Proof of Theorem~\ref{thm:notapprox}}\label{sec:nonapprox}
In the rest of the paper, for the sake of exposition we are going to assume $n=4$ and $m=2$, though all results generalize to any $n-2\geq m\geq 2$, as explained in Remark~\ref{rmk:proj-2}. We recall that we will often identify ${\bigwedge^2\R^4}$  with the  Grassmannian $\widetilde{ Gr}(2,4)$. {We recall that } we denote by $e_{ij}$ the elements $e_i\wedge e_j$ of the canonical basis of {$\bigwedge^2\R^4$, where $e_1,\dots,e_4$ is the canonical basis of $\R^4$.}

The goal of this section is to prove Theorem~\ref{thm:notapprox}.
The proof is inspired by \cite[Theorem~4]{BurIva}, where Burago and Ivanov show how to construct a measure $\mu$ which cannot be approximated by means of smooth surfaces. 
We instead study the corresponding question with approximation by means of Lipschitz multigraphs. 

The measure will be defined as the sum of three atoms. The choice of the supporting $2$-vectors   is rather delicate, as we need them to satisfy several constraints: the barycenter must be parallel to $e_{12}$, the scalar products with $e_{12}$ need to be strictly positive, and they all need to have the same norm in order to be able to use the identification of ${\bigwedge^2\R^4}$ with $\widetilde{ Gr}(2,4)$. Indeed, this identification requires the $2$-vectors to be unit length, and in order to preserve the property of the barycenter, we need to renormalize the $2$-vectors with the same constant. 

On top of these requirements, in order to use the measure $\mu$ in combination with Proposition~\ref{prop:lower_bd} to prove Theorem~\ref{thm:notapprox}, we need to find a linear transformation of $\R^4$ that preserves the barycenter property and the positive orientation, but maps two of the $2$-vectors to be almost parallel to $e_{12}$ and one of them to be almost parallel to $e_{43}$. With all these constraints in mind, we make the following Ansatz, considering the simple 2-vectors:
\begin{equation}\label{eq:def_v}
    \begin{split}
        v_1 &=\bigl(\,\delta e_1 + \tfrac12\bigl(e_3+\varepsilon^2\delta e_2\bigr)\bigr)\wedge
      \bigl(\,\delta e_2 + \tfrac12\bigl(e_4+\varepsilon^2\delta e_1\bigr)\bigr),\\
v_2 &=\bigl(\,\delta e_1 - \tfrac12\bigl(e_3+\varepsilon^2\delta e_2\bigr)\bigr)\wedge
      \bigl(\,\delta e_2 - \tfrac12\bigl(e_4+\varepsilon^2\delta e_1\bigr)\bigr),\\
v_3 &=\tfrac12\bigl(e_4+\varepsilon^2\delta e_1\bigr)\wedge
      \bigl(e_3+\varepsilon^2\delta e_2\bigr),
    \end{split}
\end{equation}
where the parameters $\e,\delta>0$ will be chosen later, and we assume $\e$ to be small.
Note that, for every $\e,\delta>0$, the barycenter condition $v_1+v_2+v_3= 2\delta^2 e_{12}$ is satisfied, since 
\[
v_1+v_2 = 2\delta^2 e_{12} + \tfrac12 (e_3 +\e^2 \delta e_2)\wedge ( e_4 + \e^2\delta e_1) = 2\delta^2 e_{12} -v_3.
\]
One can also check by an explicit calculation that {there exist $X_i\in\R^{2\times 2}$ and $c_i>0$ such that it holds 
\begin{equation}\label{eq:def-Xi}
	v_i=c_i \bigwedge M(X_i), \quad \text{for }i=1,2,3.
\end{equation}}
We check this on $v_1$, since for $v_2$ the computation is analogous and for $v_3$ it is trivial. We have 
\[
v_1 = \delta^2(1-\tfrac{\e^4}4) e_{12}  + \tfrac14 \e^2\delta (-e_{13}+e_{24}) +  \tfrac12 \delta  (e_{14} - e_{23})   + \tfrac 14 e_{34}. 
\]
We define
\[
X_1=\frac{1 }{\delta(4-\e^4)}\begin{pmatrix}
    2 & -\e^2 \\
    -\e^2  & 2
\end{pmatrix},
\]
so that computing $\bigwedge M(X_1)$ yields 
\[
    \bigwedge M(X_1) = e_{12} +\tfrac{\e^2 }{\delta(4-\e^4)} (-e_{13}+ e_{24})+ \tfrac{2}{\delta(4-\e^4)} (e_{14}-e_{23}) + \tfrac{ 1 }{\delta^2(4-\e^4)} e_{34},
\]
which coincides with $\tfrac{4 }{\delta^2(4-\e^4)} v_1$.
We now aim to find $\delta$ as a function of $\varepsilon$ such that:
\[
\|v_1\| = \|v_2\| = \|v_3\|,
\]
where the norm is the Euclidean norm induced on ${\bigwedge^2\mathbb{R}^4}$ by the standard orthonormal basis $e_{ij}$ {(introduced at the beginning of the section)}.
Since the expressions of $v_1$ and $v_2$ differ only by the sign of symmetric terms, we have $
\|v_1\| = \|v_2\|$. 
Moreover we can explicitly compute:
\begin{equation}\label{norm1}
    \lVert v_1\rVert^{2}
  =\Bigl(1-\tfrac{\varepsilon^{4}}{4}\Bigr)^{2}\delta^{4}+\Bigl(\tfrac12+\tfrac{\varepsilon^{4}}{8}\Bigr)\delta^{2}
    +\tfrac1{16},
\end{equation}
and
\begin{equation}\label{norm3}
\lVert v_3\rVert^{2}
 =\tfrac14\bigl(\varepsilon^{8}\delta^{4}+2\varepsilon^{4}\delta^{2}+1\bigr).
\end{equation}
We can now equate \eqref{norm1} with \eqref{norm3}, and letting $t = \delta^2$, we get the quadratic equation:
\[
(16-8\varepsilon^{4}-3\varepsilon^{8})\,t^{2}
\;+\;(8-6\varepsilon^{4})\,t
\;-\;3=0,
\]
which admits two distinct real solutions if and only if the discriminant 
$$
\Delta=64\bigl(4-3\varepsilon^{4}\bigr)>0,
\qquad \Leftrightarrow \qquad 
0<\varepsilon<\left(\tfrac43\right)^{\frac 14}.
$$
On the same interval the quadratic coefficient
\(16-8\varepsilon^{4}-3\varepsilon^{8}\) is positive, hence we have exactly one positive root.

We conclude that, for every $0<\varepsilon<\left(\tfrac43\right)^{1/4}$, there exists a unique $\delta(\e) > 0$ such that $\|v_1\| = \|v_2\|  =\|v_3\|$ and the value is:
\[
    \delta(\e) = \sqrt{\frac{6\varepsilon^{4}-8
      +\sqrt{(8-6\varepsilon^{4})^{2}
             +12\bigl(16-8\varepsilon^{4}-3\varepsilon^{8}\bigr)}}
     {2\bigl(16-8\varepsilon^{4}-3\varepsilon^{8}\bigr)}
    }= \sqrt{\frac{-8 + 16 + O(\e^4)}{32+O(\e^4)}}= \frac12 + O(\e^4).
  \] 

As mentioned above, since $v_1$ and $v_2$ are not almost parallel to $e_{12}$, we will consider the linear transformation $R=diag(1,1,\e,\e)$, that maps $v_1$,$v_2$,$v_3$ respectively to
\begin{equation}\label{eq:def_w}
    \begin{split}
        w_1 &=\bigl(\,\delta e_1 + \tfrac12\bigl(\e e_3+\varepsilon^2\delta e_2\bigr)\bigr)\wedge
      \bigl(\,\delta e_2 + \tfrac12\bigl(\e e_4+\varepsilon^2\delta e_1\bigr)\bigr),\\
w_2 &=\bigl(\,\delta e_1 - \tfrac12\bigl(\e e_3+\varepsilon^2\delta e_2\bigr)\bigr)\wedge
      \bigl(\,\delta e_2 - \tfrac12\bigl(\e e_4+\varepsilon^2\delta e_1\bigr)\bigr),\\
w_3 &=\tfrac12\bigl(\e e_4+\varepsilon^2\delta e_1\bigr)\wedge
      \bigl(\e e_3+\varepsilon^2\delta e_2\bigr).
    \end{split}
\end{equation}
It is straightforward to compute the following norms and scalar product, which are going to be useful:
\begin{align}
    \langle w_1, e_{12} \rangle &=\delta^{2}\!\left(1-\frac{\varepsilon^{4}}{4}\right)= \frac{1}{4}+O(\e^4), \label{eq1}\\[2ex]
    \|w_{1}\|^{2}
   &=\left(1-\frac{\varepsilon^{4}}{4}\right)^{2}\delta^{4}
     \;+\;\Bigl(\frac{\varepsilon^2}{2}+\frac{\varepsilon^{6}}{8}\Bigr)\delta^{2}
     \;+\;\frac{\varepsilon^{4}}{16} = \frac{1}{16} + O(\e^2),   \label{eq2}\\[2ex]
     \langle w_3, e_{34} \rangle \;&=\; -\frac{\varepsilon^2}{2}, \quad  
\|w_3\|^2 \;=\;\frac14\bigl(\varepsilon^2+\varepsilon^{4}\delta^{2}\bigr)^{2}= \frac{\varepsilon^{4}}{4} + O(\e^6).   \label{eq3}
\end{align}

\begin{definition}\label{def:e}
 The oriented plane associated to $e_{12}$, respectively $e_{43}$, will be called horizontal (or $h$-plane), respectively  vertical (or $v$-plane). We let $P_H, P_V$ denote the orthogonal projection on the $h$-plane, $v$-plane respectively. {Given $\e>0$, an oriented plane}   $\pi\in \widetilde{Gr}(2,4)$ is said to be $\e$-horizontal (resp. $\e$-vertical) if $P_H\llcorner \pi $ (resp. $P_V\llcorner \pi$) is an orientation-preserving map, and its inverse is $(1+\e)$-Lipschitz.
\end{definition}

We explicitly observe that an $\e$-horizontal (resp. $\e$-vertical) oriented plane is positively oriented with respect to the oriented plane $e_{12}$ (resp. the oriented plane $e_{43}$). {Note that this is equivalent to say that the scalar product of the $2$-vectors associated to such planes with $e_{12}$ (resp. with $e_{43}$) is positive.}

The following lemma will be useful for checking Definition~\ref{def:e} on the $2$-vectors $w_i$ defined in \eqref{eq:def_w}:
\begin{lemma}\label{lem:horver}
Let \(v\) be a simple 2-vector in {\(\bigwedge^2\mathbb{R}^4\)} associated to an oriented plane $\pi$.

 If $v$ satisfies \begin{equation}\label{scalar}
        \langle v, e_{12} \rangle \geq (1 + \varepsilon)^{-1}\|v\|.
    \end{equation}
    then $\pi$ is $\e$-horizontal.     
    Analogously, if $v$ satisfies
    \begin{equation}\label{scalar2}
        \langle v, e_{43}\rangle \geq (1 + \varepsilon)^{-1}\|v\|,
    \end{equation}
    then $\pi$ is $\e$-vertical.
    \end{lemma}

\begin{proof}

 Without loss of generality we assume $\|v\|=1$. We will prove just the first statement, as the second one is equivalent, up to a rotation.

Let $\theta_{1}\le\theta_{2}\in[0,\pi/2)$ be the principal angles between the
oriented planes $v$ and $e_{12}$. 
Since \(v\) and \(e_{12}\) are unit simple 2-vectors, we have:
\begin{equation}\label{principal}
   \langle v, e_{12} \rangle = \cos \theta_1 \cos \theta_2,
\end{equation}
which combined with \eqref{scalar} implies that $\cos\theta_{1}\cos\theta_{2}\ge (1+\varepsilon)^{-1}.$ Recalling that $\cos(\theta_i)\in [0,1]$ for $i=1,2$, this in turn implies that $\cos(\theta_i)\ge (1+\varepsilon)^{-1}$  for $i=1,2$.

In an appropriate orthonormal basis the
restriction $P_H\llcorner \pi$ is represented by the diagonal matrix 
\[
\begin{pmatrix}\cos\theta_{1}&0\\0&\cos\theta_{2}\end{pmatrix}.
\]
Because $\cos\theta_{i}>0$, $P_H\llcorner \pi$ is invertible. Moreover
$
\|(P_H\llcorner \pi)^{-1}\|
    ={\min\{\cos\theta_{1},\cos\theta_{2}\}}^{-1}\leq 1+\e.
$
\end{proof}

Thanks to Lemma~\ref{lem:horver}, it is easy to observe that, for $\e$ small, the oriented plane associated with $w_3$ is $\e$-vertical, while those associated with $w_1,w_2$ are $\e$-horizontal. Indeed, by  \eqref{eq1},\eqref{eq2} and \eqref{eq3}, it holds 
\begin{equation}\label{orver}
    \langle w_3, e_{43}\rangle=\frac{\e^2}2 
\ge (1 + \varepsilon)^{-1}\|w_3\|, \quad 
\langle w_1, e_{12}\rangle= \frac{1}{4}+O(\e^4)
\ge (1 + \varepsilon)^{-1}\|w_1\|.
\end{equation}

\begin{remark}\label{rmk:proj}
    Given an $\e$-horizontal plane $\pi_1$, since the map $(P_H\llcorner \pi_1)^{-1}$
    is $(1+\e)$-Lipschitz, then the Jacobian satisfies $J(P_H\llcorner \pi_1)^{-1}\leq (1+\e)^2$.
    
    Given an $\e$-vertical plane $\pi_2$, then the map $P_H\llcorner \pi_2$ is $\sqrt{3\e}$-Lipschitz, and hence has Jacobian bounded by $3\e$. Indeed, for every \(x\in \pi_2\) we use the orthogonal decomposition
\begin{equation}\label{eq:PQdecomp}
|x|^{2}=|P_Hx|^{2}+|P_Vx|^{2}.
\end{equation}
Since $\pi_2$ is $\e$-vertical, we have that $(P_V\llcorner \pi_2)^{-1}$
    is $(1+\e)$-Lipschitz, which implies
\begin{equation}\label{eq:PxLower}
|P_Vx|\;\ge\;\frac{|x|}{1+\varepsilon}.
\end{equation}
Inserting \eqref{eq:PxLower} into \eqref{eq:PQdecomp} we obtain the desired bound:
\[
|P_Hx|^{2}
\;\le\;
|x|^{2}\Bigl(1-\tfrac1{(1+\varepsilon)^{2}}\Bigr)
\;=\;
|x|^{2}\cdot\frac{\varepsilon(2+\varepsilon)}{(1+\varepsilon)^{2}}\;\le\; 3\e |x|^2.
\]

Analogously, the map $(P_V\llcorner \pi_2)^{-1}$
    is $(1+\e)$-Lipschitz, hence the Jacobian satisfies $J(P_V\llcorner \pi_2)^{-1}\leq (1+\e)^2$, and the map $P_V\llcorner \pi_1$ is $\sqrt{3\e}$-Lipschitz, and hence has Jacobian bounded by $3\e$. 
\end{remark}

We now fix $f\in Lip(B,\A(\R^2))$ with $f\llcorner\bd B \equiv Q[0]$. 
We consider $\textbf{T}_f$ {as defined in Section \ref{sub:Q-val maps}} and in this section we drop the subscript $f$ in $\textbf{T}_f$ as the function $f$ will be fixed. 
Let spt$\mathbf{T}\subseteq \R^4$ denote the support of $\mathbf{T}$. If needed, we extend $f$ on $\R^2\setminus B $ as $Q[0]$, and consequently $\mathbf{T}$.
Let $T_0$ denote the subset of points $p\in \text{spt} \mathbf{T}$  of integer density for $\mathbf{T}$ admitting a tangent plane.

Note that $|(\text{spt}\mathbf{T})\setminus T_0|=0$ since $f$ is Lipschitz, see \cite[Theorem 1.13]{DeLSpa11}, and $T_p T_0$ is well-defined for every $p\in T_0$. 
We define  
\begin{equation}\label{parti}   
    \begin{split}
        &T_h:=\{ p\in T_0 : T_p T_0\ \text{is }  \e\text{-horizontal} \}, \quad T_v:=\{ p\in T_0 : T_p T_0\ \text{is }  \e\text{-vertical}\},\\
        &\mathbf{T}^H:= \textbf{T}\llcorner T_h, \qquad \mathbf{T}^V:=\textbf{T}\llcorner T_v, \qquad \mathbf{T}^M:=\mathbf{T}-\mathbf{T}^H- \mathbf{T}^V.
    \end{split}     
    \end{equation}
We notice that $\mathbf{T}^H$ and $\mathbf{T}^V$ are well-defined, since $\mathbf{T}$ is a rectifiable current and $T_h$ and ${T_v}$ are measurable subsets of $T_0$. Moreover, by definition the tangent planes in $\textbf{T}^V$ are positively oriented with respect to  the $v$-plane. 
Given $p\in \R^4$ and $r>0$, we also denote  $\textbf{B}^\square_r(p)=\mathbf{T}^\square\llcorner B_r(p)$, where the symbol $\square\in\{V,M,H\}$.

{We remark that since $f$ is Lipschitz, then $\text{spt}\mathbf{T}$ (and consequently $\text{spt}\mathbf{T}^V$ and $\text{spt}\mathbf{T}^H$) are bounded. This follows since $\G(f(0),f(x))\le Lip(f)$ for every $x\in B$, and thus  $\text{spt}(f(x))$ is contained in a fixed ball of $\R^2$ depending only on $f(0)$ and $Lip(f)$.}

\begin{lemma}\label{lem:Q_dens_vert}
    There exists a universal constant $\beta>0$ with the following property. For $\e=\e(Q)>0$ small enough in \eqref{parti}, if $\M(\textbf{B}_r^V(p))\ge \beta r^2$ and $\M(\textbf{B}_{2r}^V(p))\le \beta (2r)^2$, then $\M(\textbf{B}_{2r}^M(p))\ge \frac \beta2 r^2$. 
\end{lemma}

Assuming Lemma~\ref{lem:Q_dens_vert}, we prove a global estimate on the mass of the vertical part.

\begin{proposition}\label{prop:lower_bd}
    For $\e=\e(Q)>0$ small enough in \eqref{parti}, there exists a universal constant  $C>0$ such that  $\M(\mathbf{T}^M)\ge C \M(\mathbf{T}^V).$
\end{proposition}

\begin{proof}
	We can assume without loss of generality that $\M(\mathbf{T}^V)>0$.  
    The proof follows from a covering argument and Lemma~\ref{lem:Q_dens_vert}.
    Fix $\beta,\e>0$ small as in Lemma~\ref{lem:Q_dens_vert}. 
    For every $p\in  T_v$   
    let $r(p)$ be the supremum of the radii $r$ such that $\M(\textbf{B}^V_{r}(p))\ge \beta r^2$. 
    Note that $r(p)$ is uniformly bounded from above (since spt$\textbf{T}^V$ is bounded as $f$ is Lipschitz). Moreover, by Lebesgue differentiation theorem, there exists $T_v'\subseteq T_v$ such that $\M(\textbf{T}^V\llcorner (T_v\setminus T_v'))=0$ and $r(p)>0$ for every $p\in T_v'$.

    By Vitali's 5-times covering lemma \cite[Lemma 3.4]{Sim-book} applied to the family $\{B_{2r(p)}(p) :p\in  T_v'\}$, 
    we find a countable set  $\mathcal I$ of indices and a sub-family $\{ B_{2r_i}(p_i) \}_{i\in \mathcal{I}}$ of pairwise disjoint balls, where $r_i:=r(p_i)$, such that  
    \begin{equation}\label{eq:cover}
         \bigcup_{p\in  T_v'} B_{2r(p)}(p)\subseteq\bigcup_{i\in \mathcal{I}} B_{10r_i}(p_i).  
    \end{equation}
     By definition of $r(p)$ it holds   $\M(\textbf{B}^V_{\rho}(p))\le \beta \rho^2$ for every $\rho>r(p)$. In particular,  we  can apply Lemma~\ref{lem:Q_dens_vert}  to the balls $\textbf{B}^V_{r_i}(p_i)$ and obtain 
    \[
    \M(\textbf{B}^M_{2r_i}(p_i))\ge \frac \beta 2 r_i^2\ge \frac1{200} \M(\textbf{B}^V_{10 r_i}(p_i)),
    \]
    where the second inequality follows from the choice of the radii in the cover.
    Using \eqref{eq:cover} we conclude $    \M(\mathbf{T}^M)\ge  \frac1{200} \M(\mathbf{T}^V ).$ 
\end{proof}

\begin{proof}[Proof of Lemma~\ref{lem:Q_dens_vert}] 
     We assume wlog $p=0$. Assume that $\M(\textbf{B}_r^V)\ge \beta r^2$ and $\M(\textbf{B}_{2r}^V)\le \beta (2r)^2$. Assume by contradiction that 
    \begin{equation}\label{eq:contr_hp}
        \M(\textbf{B}_{2r}^M)\le \frac \beta 2 r^2.
    \end{equation}
    \noindent\textit{Step 1: We estimate from above the mass of the horizontal part of the graph. }

     We start recalling that the projection $P_H$ restricted to an $\e$-horizontal plane is a bijection with inverse being an orientation-preserving $(1+\e)$-Lipschitz map, see Definition~\ref{def:e}.

     Fix $\rho\in(r,2r)$. By the formula for the pushforward of a current \cite{Sim-book}, for every $\eta \in \mathscr D^2(\R^2)$ it holds 
    \begin{equation}\label{eq:form_pushf}
        (P_H)_\sharp \textbf{B}^H_\rho(\eta) =  \int_{T_h\cap B_\rho} \langle \eta(P_H(x)), (dP_H(x))_\sharp \omega(x)  \rangle d \mu_{\textbf{T}}(x),
    \end{equation}
    where $\omega$ is the orientation of $\textbf{T}$ and  $\mu_\textbf{T}$ is the surface measure associated to the current $\textbf{T}$. Since on $T_h$ the tangent planes to $\textbf{T}$ are all $\e$-horizontal, one can estimate thanks to Remark~\ref{rmk:proj} $|(dP_H(x))_\sharp \omega(x)|\ge (1+\e)^{-2}\ge 1-2\e$ for every $x\in T_h$. Thus, since by duality $\sup_{|\eta|\le 1, \eta\in\mathscr D^2(\R^2)} \langle \eta, \omega\rangle = |\omega| $ for every vector field $\omega:\R^2\to {\bigwedge^2\R^2}$, by the previous estimate and taking the   $\sup_{|\eta|\le 1, \eta \in \mathscr D^2(\R^2)}$ in \eqref{eq:form_pushf} we get   
    \[
    \M((P_H)_\sharp \textbf{B}^H_\rho) \ge   (1-2\e)\int_{T_h \cap B_\rho}   d \mu_{\textbf{T}}(x) = (1-2\e)\M(\textbf{B}^H_\rho).
    \]
    Note that we do not have cancellations in $(P_H)_\sharp \textbf{B}^H_\rho$, as all the tangent planes in $\textbf{T}^H$ are positively oriented with respect to  the $h$-plane. Recalling that $\textbf{T}$ is the current associated to the $Q$-valued graph of $f$, we deduce that $\textbf{B}^H_\rho$ projects onto the $h$-plane with multiplicity at most $Q$. We conclude that 
    \begin{equation}\label{eq:Q_upper_bd}
    Q\pi\rho^2 \ge \M((P_H)_\sharp \textbf{B}^H_\rho) \ge   (1-2\e)\M(\textbf{B}^H_\rho).        
    \end{equation}

\noindent\textit{Step 2: We estimate from below the mass of the vertical projection of the graph. }

    By the assumption $\M(\textbf{B}_r^V)\ge \beta r^2$, we can argue as  in Step 1 to show that 
    for every $\rho\in [r,2r]$ it holds 
    \begin{equation}\label{eq:est_V_below}
        \M((P_{V})_\sharp  \textbf{B}_\rho^V) \ge \frac 34 \beta r^2.
    \end{equation}
    Indeed,   we recall   that the projection $P_V$ restricted to an $\e$-vertical plane is a bijection with inverse being an orientation-preserving $(1+\e)$-Lipschitz map, see Definition~\ref{def:e}.  By the formula for the pushforward of a current we have
    \[
    (P_V)_\sharp \textbf{B}^V_\rho(\eta) =  \int_{T_v\cap B_\rho} \langle \eta(P_V(x)), (dP_V(x))_\sharp \omega(x)  \rangle d \mu_{\textbf{T}}(x),
    \]
    where $\omega,\mu_\textbf{T}$ are as above. Since on $T_v$ the tangent planes to $\textbf{T}$ are all $\e$-vertical, by  Remark~\ref{rmk:proj} one can estimate   $|(dP_V(x))_\sharp \omega(x)|\ge (1+\e)^{-2}\ge 1-2\e$ for every $x\in T_v$. Thus, taking the $\sup$ over forms in $\mathscr D^2(\R^2)$ with norm bounded by 1 and using the assumption of the Lemma, we conclude \eqref{eq:est_V_below} as follows:
    \[
    \M((P_V)_\sharp \textbf{B}^V_\rho) \ge   (1-2\e)\M(\textbf{B}^V_\rho) \ge (1-2\e) \beta r^2\ge \frac34 \beta r^2
    \] 
    by taking $\e $ smaller than a universal constant.
    
    Furthermore, we have 
    \[
    \begin{split}
    & \M((P_{V})_\sharp  \textbf{B}_\rho^M)\le \M(\textbf{B}^M_{\rho})\le \M(\textbf{B}^M_{2r})\le \frac \beta {2} r^2  \\
    &  \M((P_{V})_\sharp \textbf{B}_\rho^H)\le 3\e \M(\textbf{B}^H_\rho)\le  3\e\frac{Q\pi (2r)^2}{1-2\e}  \le 24\pi Q \e  r^2, 
    \end{split}
    \]
    where in the last line we considered  $\e\le 1/4$  and used \eqref{eq:Q_upper_bd}, the formula for the mass of the push-forward and that the Jacobian  of $ P_V$ restricted to a $\e$-horizontal plane is bounded by $3\e$, see Remark~\ref{rmk:proj}. This, combined with $\mathbf{T}\llcorner B_\rho-\textbf{B}_\rho^M-\textbf{B}_\rho^H=\textbf{B}_\rho^V$, the subadditivity of the mass and the linearity of the pushforward,  yields 
    \begin{equation}\label{eq:bound_M}
    \M((P_{V})_\sharp  (\mathbf{T}\llcorner B_\rho)) \ge    \M((P_{V})_\sharp  \textbf{B}_\rho^V) - \M( (P_V)_\sharp \textbf{B}_\rho^M) - \M((P_{V})_\sharp  \textbf{B}_\rho^H)\ge r^2(\tfrac \beta 4-24\pi Q\e),        
    \end{equation}  
    and the right-hand side is positive taking $\e$ small depending on $Q$. 

\noindent\textit{Step 3: We bound from below the mass of the mixed part, deriving a contradiction with \eqref{eq:contr_hp}.}

     For a.e. $\rho \in (r,2r)$, we have that $\M(\bd (P_V)_\sharp  ( \mathbf{T}\llcorner B_\rho))$ is finite. 
     This, combined with the fact that $(P_V)_\sharp  ( \mathbf{T}\llcorner B_\rho)$ is a $2$-dimensional integer {multiplicity} current in the $2$-dimensional $v$-plane, allows to apply the decomposition result \cite[Theorem  27.6]{Sim-book} to $(P_V)_\sharp  ( \mathbf{T}\llcorner B_\rho)$.

     We consider the current $(P_V)_\sharp  ( \mathbf{T}\llcorner B_\rho)$ {with orientation $\omega(x)$ given by the induced orientation of  $(P_V)_\sharp  ( T_x\mathbf{T})$ for $x\in \text{spt}\bm T\llcorner B_\rho$}. Denoting $\tilde \theta $ the (non-negative) density of such current, we define $\theta$ as the signed multiplicity of the current with orientation fixed to $e_{43}$, i.e. $\theta=\tilde \theta \, {\omega \cdot e_{43}}$, and we extend it to zero on the whole $v$-plane outside spt$(P_V)_\sharp  ( \mathbf{T}\llcorner B_\rho)$.
     We consider the following $2$-dimensional finite perimeter sets in the $v$-plane:
    \[
    U_j:=\{\theta \ge j\}, \quad  V_j:=B_R\setminus U_j=\{\theta \le j-1\}, \quad \mbox{ for $j\in\Z$},
    \]
    where $B_R$ is a large ball containing the support of $(P_V)_\sharp  ( \mathbf{T}\llcorner B_\rho)$.
    Hence $\partial U_j\subset\joinrel \subset B_R$ for every $j\geq 1$ and $\partial V_j\subset\joinrel \subset B_R$ for every $j\leq 0$. Moreover, up to choosing $R$ even bigger, we can guarantee that
    $$
      |U_j|\le |V_j| \; \mbox{ for every $j\ge 1$, \quad \qquad and \quad \qquad }  |U_j|\ge |V_j| \; \mbox{ for every $j\le 0$}.
    $$
    In particular, denoting with $P_j=P(U_j,B_R)=P(V_j,B_R)$ the relative perimeter of the set $U_j$ in $B_R$, we conclude from the isoperimetric inequality that there exists a universal constant $c>0$ such that
\begin{equation}\label{bigiso}
 P_j^2\ge c|U_j| \; \mbox{ for every $j\ge 1$, \quad \qquad and \quad \qquad }  P_j^2\ge c|V_j| \; \mbox{ for every $j\le 0$}.
\end{equation}
    By \cite[Theorem~27.6]{Sim-book}  it holds: 
    $$
    \bd ((P_V)_\sharp  ( \mathbf{T}\llcorner B_\rho))=\sum_{j\geq 1} \bd \llbracket U_j\rrbracket - \sum_{j\leq 0} \bd \llbracket V_j\rrbracket, \quad \mbox{ and } \quad \M(\bd (P_V)_\sharp  ( \mathbf{T}\llcorner B_\rho))=\sum_{j\in \Z} P_j,
    $$ 
 Moreover\footnote{We recall that the characteristic function $\chi_E$ of a set $E$ is defined as $\chi_E(x)=1$ if $x\in E$ and $\chi_E(x)=0$ if $x\notin E$.}
 $$\M((P_V)_\sharp  ( \mathbf{T}\llcorner B_\rho))=\int_{B_R}\sum_{j\ge 1}\chi_{\{\theta\ge j\}} +\chi_{\{\theta\le -j\}}\ud \mathcal H^2=\sum_{j\ge 1}|U_j|+\sum_{j\le 0}|V_j|.$$ 
 Combining the previous two displayed equations with \eqref{bigiso}, we thus estimate
    \begin{equation}\label{eq:per_bd}
        \M(\bd (P_V)_\sharp  ( \mathbf{T}\llcorner B_\rho))^2 \ge\sum_{j\in\Z} P_j^2 \ge c \left( \sum_{j\ge 1}  |U_j| + \sum_{j\le 0} |V_j| \right) =  c \,\M((P_V)_\sharp  ( \mathbf{T}\llcorner B_\rho)).
    \end{equation}

    Let $L^\square_\rho(p)= \M(\langle \mathbf{T}^\square,\ud_p,\rho\rangle)$, for $\square\in\{H,V,M\}$, be the mass of the slice of $\textbf{T}^\square$ by the (Euclidean) distance function to $p$. By the coarea formula \cite[Lemma  28.5]{Sim-book}
    it holds 
    \begin{equation}\label{coarea}
        \int_r^{2r} L^\square_\rho(p)\ud\rho \leq  \M(\textbf{B}^\square_{2r}(p)).
    \end{equation} 
    Moreover we explicitly observe that, since $\mathbf{T}=\mathbf{T}_f$, then $\partial \mathbf{T}=Q\llbracket \partial B\rrbracket$, which is entirely contained in the $h$-plane. Hence $(P_{V})_\sharp (\bd \mathbf{T})=0$ and we deduce that
    \begin{equation}\label{bdr}
        (P_{V})_\sharp (\bd (\mathbf{T}\llcorner B_\rho))=(P_{V})_\sharp \big( (\bd \mathbf{T})\llcorner B_\rho)+\langle \mathbf{T},\ud_p,\rho\rangle \big)=(P_{V})_\sharp \langle \mathbf{T},\ud_p,\rho\rangle.
    \end{equation}
    Combining \eqref{eq:bound_M}, \eqref{eq:per_bd},  \eqref{bdr}, and using the subadditivity of the mass, we get
    \begin{multline}\label{usefulhighcod}
        r \sqrt{c(\beta/4-24Q\pi \e)} \le\M( \bd {P_V}_\sharp (\mathbf{T}\llcorner B_\rho) ) =\M((P_{V})_\sharp (\bd (\mathbf{T}\llcorner B_\rho)))=\M((P_V)_\sharp  \langle \mathbf{T}, \ud_p,\rho \rangle ) \\
        \le \sum_{\square\in\{V,H,M\}} \M((P_V)_\sharp  \langle \mathbf{T}^\square, \ud_p,\rho \rangle ) 
       \le  L^V_\rho+ L^M_\rho + \sqrt{3\e} L^H_\rho,
    \end{multline}
    where in the last line we used again the fact that $P_V$ restricted to an $\e$-horizontal plane is $\sqrt{3\e}$-Lipschitz (as proved in Remark~\ref{rmk:proj}) and {argued} as in \eqref{eq:Q_upper_bd}.
    Integrating this inequality  from $r$ to $2r$ and using  \eqref{coarea}, we get 
    \[  
    \M(\textbf{B}^V_{2r}) + \M(\textbf{B}^M_{2r}) + \sqrt{3\e} \M(\textbf{B}^H_{2r})\ge r^2 \sqrt{c(\beta/4-24Q\pi\e)}. 
    \]
    On the other hand, by \eqref{eq:Q_upper_bd} we have $ \M(\textbf{B}^H_{2r})\le 8\pi Q r^2$, and by assumption $\M(\textbf{B}^V_{2r})\le 4\beta r^2.$ Therefore, choosing $\e=\e(Q)>0$ small enough dependently on $Q$, and $\beta>0$ small enough (independent on $Q$), we arrive at 
    \[  
    \M(\textbf{B}^M_{2r})\ge r^2 (\sqrt{c(\beta/4-24\pi Q\e)}-8\pi Q\sqrt{3\e}-4\beta) > \beta r^2,
    \]
    which contradicts \eqref{eq:contr_hp} and thus concludes the proof.
\end{proof}

\begin{remark}\label{rmk:general_bdry}
In view of the proof of Theorem~\ref{thm:qcvx_non_poly}, we explicitely observe that Lemma~\ref{lem:Q_dens_vert} and Proposition~\ref{prop:lower_bd} hold, with
identical proofs, for the current $\textbf{T}_f$ associated to any
$f\in Lip(\Omega,\A(\R^2))$ on a bounded Lipschitz domain
$\Omega\subseteq\R^2$ with constant boundary trace
$f\llcorner\bd \Omega\equiv\sum_{j=1}^J q_j[a_j]$, where
$\sum_{j=1}^J q_j[a_j]\in\A(\R^2)$. Indeed, the boundary datum enters the
proof only through \eqref{bdr}, and
\[
\bd\textbf{T}_f=\sum_{j=1}^J q_j\,\llbracket \bd \Omega\times\{a_j\}\rrbracket
\]
is supported in finitely many horizontal translates
$\R^2\times\{a_j\}$ of the $h$-plane, so that
$(P_V)_\sharp(\bd\textbf{T}_f)=0$ still holds. Moreover, the shape of the
domain plays no role: \eqref{eq:Q_upper_bd} only uses that $\textbf{T}_f$
projects onto the $h$-plane with multiplicity at most $Q$, and
$\text{spt}\,\textbf{T}_f$ is bounded since $f$ is Lipschitz on a bounded
domain. 
\end{remark}

\begin{remark}\label{rmk:proj-2}
    We remark that Theorem~\ref{thm:notapprox}, and hence the consequent Theorem~\ref{thm:qcvx_non_poly},
 can be generalized to elements in $\widetilde{Gr}(m,n)$ for $n-2\ge m\geq 2$. Indeed, in this case, for any    $\pi_1\in\widetilde{Gr}(m,n)$ there exists an $m$-plane $\pi_2\in\widetilde{Gr}(m,n)$ such that $\pi_1\cap \pi_2$ is $(m-2)$-dimensional and the orthogonal complement of $\pi_1\cap \pi_2$ in $\pi_2$ is contained in $\pi_1^\perp$. Thus, by replacing the $h$-plane and $v$-plane with $\pi_1,\pi_2$, one can still obtain \eqref{bdr}, and one can show analogous bounds for the Jacobians as in Remark~\ref{rmk:proj}.
 Although the Lipschitz bounds of the vertical projection of an $\e$-horizontal plane will hold just on a $2$-dimensional subspace of the $\e$-horizontal plane, this is enough to deduce the required smallness of the Jacobian and in the estimate \eqref{usefulhighcod}. For the latter, we observe that $\langle \textbf{T}^H,\text{d}_p,\rho \rangle $ will now be an $(m-1)$-dimensional integer current, hence the tangential Jacobian will always be small, as every tangent $(m-1)$-plane of $\langle \textbf{T}^H,\text{d}_p,\rho \rangle $ will have at least one direction which is orthogonal to $\pi_1\cap \pi_2$.
 On the contrary, this  argument clearly fails in the codimension-one case  $\widetilde{Gr}(n-1,n)$.
 This is not a surprise, as in codimension one Minkowski existence Theorem \cite{Sch} guarantees that weighted Gaussian images of Lipschitz graphs are dense in the space of positive measures on $\widetilde{Gr}(n-1,n)$. Hence Theorem~\ref{thm:notapprox} fails, and the notions of polyconvexity and quasiconvexity coincide with the notion of convexity \cite{Dac-book}. In particular, Theorem~\ref{thm:qcvx_non_poly} fails as well. 
\end{remark}

With the previous results in hand we are able to prove Theorem~\ref{thm:notapprox}.

\begin{proof}[Proof of Theorem~\ref{thm:notapprox}]
    We first show a similar result for an atomic measure supported on the $2$-vectors $w_1,w_2,w_3$ defined in \eqref{eq:def_w}, and then conclude the  proof of the theorem via an affine transformation.
Throughout the proof, we will denote with $P_{v_i},P_{w_i}\in \widetilde{Gr}(2,4)$ the oriented 2-plane associated to $v_i$, $w_i$ for $i=1,2,3.$
    
    Consider the atomic measure on $\widetilde{Gr}(2,4)$ defined by
     $$
     \mu_0:=\frac{Q|B|}{2\delta^2}\sum_{i=1}^3 \|w_i\|  [P_{w_i}].
     $$
    Consider $U_V\subseteq {\widetilde{Gr}}^+(2,4)$ the interior of the set of $\e$-vertical planes and $U_M$ the complement of the $\e$-horizontal and $\e$-vertical planes. Consider the open sets
\[
W_H:=\{\pi\in\widetilde{Gr}(2,4) :\langle \pi, e_{12}\rangle>(1+\e)^{-1}\},
\qquad
W_V:=\{\pi\in\widetilde{Gr}(2,4): \langle \pi, e_{43}\rangle>(1+\e)^{-1}\},
\]
which, by Lemma~\ref{lem:horver}, consist of $\e$-horizontal and
$\e$-vertical planes respectively. Since for $\e$ small the inequalities in
\eqref{orver} are strict (for $w_3$, strictness amounts to $\e\delta^2<1$),
it holds $P_{w_1},P_{w_2}\in W_H$ and $P_{w_3}\in W_V\subseteq U_V$. In
particular $\mu_0(U_V)>0$ and, since $U_M\cap(W_H\cup W_V)=\emptyset$ and
$W_H\cup W_V$ is open, also $\overline{U_M}\cap(W_H\cup W_V)=\emptyset$,
whence $\mu_0(\overline{U_M})=0$.

For a given $Q$-valued Lipschitz map $f:B \to \mathcal A_Q(\R^2)$ with
$f\llcorner {\bd B }\equiv Q[0]$, let $\textbf{T}_f$ be the $Q$-valued graph
associated to $f$ and denote with $\gamma_{\textbf{T}_f}$ the weighted
Gaussian image of $\textbf{T}_f$. Then Proposition~\ref{prop:lower_bd}
implies that $\gamma_{\textbf{T}_f}(U_M)\ge C\gamma_{\textbf{T}_f}(U_V)$. We
deduce that $\mu_0$ cannot be approximated by Gaussian images of the form
$\gamma_{\textbf{T}_f}$. Indeed, if we had $\gamma_{\textbf{T}_{f^k}}\rightharpoonup \mu_0$
for a sequence $f^k$, then by lower (resp.\ upper) semicontinuity of the mass of open
(resp.\ closed) sets under weak-* convergence we would get the contradiction
\[
0<C\mu_0(U_V)\le C\liminf_{k\to+\infty}\gamma_{\textbf{T}_{f^k}}(U_V)
\le \liminf_{k\to+\infty}\gamma_{\textbf{T}_{f^k}}(U_M)
\le \limsup_{k\to+\infty}\gamma_{\textbf{T}_{f^k}}(\overline{U_M})
\le \mu_0(\overline{U_M})=0.
\]
Note that the renormalization $\frac{Q|B|}{2\delta^2}$ ensures that the Gaussian measures have the same barycenter as $\mu_0$.

    Then consider the atomic measure $\mu$ on $\widetilde{Gr}(2,4)$:
    $$
    \mu := \frac{Q|B|}{2\delta^2}\sum_{i=1}^3 \|v_i\|[P_{v_i}] =\frac{Q|B|}{2\delta^2}\|v_1\| \sum_{i=1}^3 [P_{v_i}].
    $$
    We aim to show that $\mu$ cannot be approximated by Gaussian images of Lipschitz $Q$-valued maps defined on $B$, which would conclude the proof of the theorem.  
     
     Let $R={diag}(1,1,\e,\e)$, identified with the linear map $\R^4\to \R^4$, which we recall satisfies $R(P_{v_i})=P_{w_i}$ for $i=1,2,3$. More precisely, recalling that $v_i=u_i^1\wedge u_i^2$, with $u_i^1,u_i^2\in \R^4$ defined in \eqref{eq:def_v}, and the definition of $w_i$ in \eqref{eq:def_w}, we have
     $$
     w_i=(Ru_i^1)\wedge (Ru_i^2).
     $$  
   This in particular implies that the tangential Jacobian of the linear map $R$ with respect to $P_{v_i}$ is:  \begin{equation}\label{tangentialjacobian}
   J_{P_{v_i}}R=\frac{\|(Ru_i^1)\wedge (Ru_i^2)\|}{\|u_i^1\wedge u_i^2\|}=\frac{\|w_i\|}{\|v_i\|}.
   \end{equation}

   In the following, we warn the reader that we will compute push-forwards sometimes in the sense of currents and sometimes in the sense of measures. There is no risk of confusion, as for every instance the correct interpretation depends on the object we push-forward. This will also determine whether we obtain tangential Jacobians in the equations.
   
   For every rectifiable current $\textbf{T}=(T,\theta\omega)$, we claim that 
    \begin{equation}\label{eq:pushf-gauss}
        \gamma_{ R_\sharp \textbf{T}} =   \, R_\sharp \gamma_{(J_TR) \textbf{T}},
    \end{equation}
    where $J_TR$ denotes the tangential Jacobian of the linear map $R$ on $T$. 
    Indeed, recall that $R_\sharp \textbf{T}=(R(T), \, (\theta\circ R^{-1}) {\tilde\omega})$, {where $\tilde\omega(Rx)$   is the orientation of the plane  $T_{Rx}R(T)=R(T_xT)$
induced by the orientation $\omega(x)$ of $T_xT$.} 
    By the definition of Gaussian image and the area formula we have 
    \begin{multline*}
        \gamma_{ R_\sharp    \textbf{T}} =  (R(P\circ R^{-1}))_\sharp ((\theta\circ R^{-1})\, \H^2\llcorner{R(T)})= (R(P\circ R^{-1}))_\sharp  (R_\sharp ((J_TR)\,  \theta\, \H^2\llcorner{T})  )=  \\
        =      R_\sharp \Big(((P\circ R^{-1})\circ R)_\sharp  ((J_TR)\, \theta\, \H^2\llcorner{T}) \Big)
        = R_\sharp P_\sharp ((J_TR)\,   \theta\, \H^2\llcorner{T}) =  R_\sharp \gamma_{(J_TR)\,  \textbf{T}}. 
    \end{multline*}
    Moreover, in the case of a current $\textbf{T}_f$ associated to the $Q$-valued graph of $f$, it holds     
    \begin{equation}\label{eq:gauss_graph}
    R_\sharp \textbf{T}_f= \textbf{T}_{\e f}
    \end{equation}
    Indeed, if 
    $f=\sum_{j=1}^Q[f_j]$ is a local selection of $f$ at $y\in B$ then $\text{spt}\textbf{T}_f$ is locally around $y$ given by  $\bigcup_{j=1}^Q \{(y,f_j(y))\}\subseteq\R^4$ and its tangent space (defined for a.e. $y\in B$) is given by the union of the planes identified by $\bigwedge M(\nabla f_j(y))$. Thus, the tangent planes in $R_\sharp\textbf{T}_f$ are locally the union of planes associated to the 2-vectors $\bigwedge M(\nabla (\e f))$ and thus one can explicitly compute $R_\sharp \textbf{T}_f$.
    
   Finally, by the very definition of push-forward, for every $P\in \widetilde{Gr}(2,4)$   it holds 
   \begin{equation}\label{eq:gauss_delta}
       R_\sharp[P]= [R(P)].
   \end{equation}   
    We now assume by contradiction that there exists a sequence of Lipschitz functions   $f^k:B \to \mathcal A_Q(\R^2)$ with $f^k\llcorner {\bd B }\equiv  Q[0]$ and such that 
    $$
   \gamma_{\textbf{T}_{f^k}}\rightharpoonup \mu = {\frac{Q|B|}{2\delta^2}}\sum_{i=1}^3 \|v_1\|[P_{v_i}],
    $$
    In turn, by \eqref{tangentialjacobian} this implies that 
    $$
    \frac{2\delta^2}{Q|B|}\gamma_{(J_{T_{f^k}}R) \, \textbf{T}_{ f^k}} \rightharpoonup   \sum_{i=1}^3 \|v_1\| \cdot (J_{P_{v_i}}R) \,  [P_{v_i}]=  \sum_{i=1}^3 \|v_i\| \cdot (J_{P_{v_i}}R) \,  [P_{v_i}]=\sum_{i=1}^3 \|w_i\|  [P_{v_i}],
    $$
    where the first convergence holds since the tangential Jacobian is continuous as the tangent plane varies continuously.
    Combining this convergence with \eqref{eq:pushf-gauss}, \eqref{eq:gauss_graph} and \eqref{eq:gauss_delta}, we deduce the following contradiction with the non-approximability of $\mu_0$:
    \begin{multline*}
         \gamma_{ \textbf{T}_{\e f^k}} =   \gamma_{ R_\sharp\textbf{T}_{f^k}} =  R_\sharp \gamma_{(J_{T_{f^k}}R) \,\textbf{T}_{f^k}} \rightharpoonup  R_\sharp \big( \tfrac{Q|B|}{2\delta^2}\sum_{i=1}^3 \|w_i\|  [P_{v_i}]\big )=\\
         = \tfrac{Q|B|}{2\delta^2}\sum_{i=1}^3 \|w_i\| R_\sharp [P_{v_i}]= \tfrac{Q|B|}{2\delta^2}\sum_{i=1}^3 \|w_i\|  [P_{w_i}]=\mu_0.
    \end{multline*} 
\end{proof}

\section{A non-polyconvex $Q$-integrand defining a lower semicontinuous energy}\label{sec:quasi/poly}
The aim of this section is to prove Theorem \ref{thm:qcvx_non_poly}. In the rest of the paper, we adopt the following convention: indices of sequences are written as superscripts (e.g., $u^k$), while indices of elements of decompositions are written as subscripts (e.g., $a_j$).

\subsection{Definition of the candidate integrand $A$}
We consider the 2-vectors $v_1,v_2,v_3$ (of equal norm) introduced in \eqref{eq:def_v} and define $\psi :\R^{2\times 2}\to \R$ as follows
\[
\begin{cases}
    \psi(X)=0\quad &\text{if $\bigwedge M(X)$ is positively proportional to $v_i$ for some $i=1,2,3,$}\\[1ex]
    \psi (X)=\|\bigwedge M(X)\| \quad &\text{otherwise.}
\end{cases}
\]
Note that $\bigwedge M(X)$ is positively proportional to $v_i$ for some $i=1,2,3,$ if and only if $\bigwedge M(X)=\lambda\bigwedge M( X_i)$ for some $i=1,2,3$ and $\lambda>0$, where the matrices $X_i$ have been {introduced in \eqref{eq:def-Xi}}.

\begin{definition}\label{def:max-dec}
    Given $p\in \A(\R^\ell)$ for $\ell\in\N$, we define a maximal decomposition of $p$ as:
    \[
    p=\sum_{j=1}^J q_j [p_j], \quad \text{with }p_j\neq p_i \text{ for every } j\neq i.
    \]
    The tuple $(q_1,\dots,q_J)$ is unique up to permutations. We  say that two points $p,p'\in\A(\R^\ell)$ have the same maximal multiplicities if they admit a maximal decomposition of the form 
    \[
    p=\sum_{j=1}^J q_j [p_j]\quad  \text{ and } \quad p'=\sum_{j=1}^J q_j[p_j'].
    \]
\end{definition}

\begin{definition}
    Consider $(a,X)\in\A(\R^2\times \R^{2\times 2})$ with maximal decomposition $ \sum_{j=1}^J q_j [(a_j,X_j)]$. For a given bounded  open set $U\subseteq \R^2$ we let
    \begin{equation}\label{def:tau-X}
    \mathcal T(a,X,U):=\sum_{j=1}^J \tau_{(a_j+X_j\,\cdot) }W^{1,\infty}_0(U,\mathcal A_{q_j}(\R^2))
    \end{equation}
    denote the family of all the functions $\varphi\in Lip( U,\A(\R^2))$ which admit a decomposition $\varphi=\sum_{j=1}^J \varphi_j$ with $\varphi_j\in Lip( U,\mathcal A_{q_j}(\R^2))$ and satisfy 
    \[
    \varphi_j\llcorner\bd U =  q_j [a_j+X_j\,\cdot ]\llcorner\bd U.
    \]
    In particular, for every $j=1,\dots,J$ it holds  $\varphi_j\in \mathcal T(q_j[(a_j,X_j)],U)$.
\end{definition}

Consider the following  $Q$-integrand,  in the sense of Definition~\ref{def:Q-integr}, defined on $\A(\R^2\times \R^{2\times 2})$:
\begin{equation}\label{eq:quasi_envel}
    A(a,X):=\inf\left\lbrace \fint_{D}\bar \psi(\nabla \varphi)\,  :\, \varphi\in \mathcal T(a,X,D)   \right\rbrace,
\end{equation} 
where we recall that $\bar \psi(Y)=\sum_{i=1}^Q \psi(Y_i)$ if $Y=\sum_{i=1}^Q [Y_i]\in \A(\R^{2\times 2})$, {and the family $\mathcal T(a,X,D)$ has been defined in \eqref{def:tau-X}.}

 Note that it is not clear whether $A$ can be written in the form  $\bar\psi'
$ for a $1$-integrand $\psi':\R^{2\times 2}\to \R$. Thus,  $A$ does not fall in the framework studied in \cite{DeRLeiYou}. 
It would be interesting to show that $A$ (or a suitable modification) falls in the framework considered in \cite{DeRLeiYou}.  We discuss this in more detail in Remark~\ref{rmk:anisotropic}.

\begin{remark}\label{rmk:nidepend}
     We  note  that the value of $A$ is invariant under translations and rescalings of the domain $D$, that is, for every $(a,X)\in\A(\R^2\times \R^{2\times 2}), x_0\in\R^2$ and $\lambda>0$ it holds 
    \begin{equation*}
        A(a, X)=\inf\left\lbrace \fint_{x_0+\lambda D}\bar \psi(\nabla \varphi)\,  :\, \varphi\in \mathcal T(a,X,x_0+\lambda D)   \right\rbrace.
    \end{equation*}    
    Indeed, if $f\in W^{1,\infty}_0(x_0+\lambda D,\A)$  then  $\tilde f:= \lambda^{-1}  f(x_0+\lambda \cdot) \in W^{1,\infty}_0(D,\A)$. A change of variable in \eqref{eq:quasi_envel}  allows to conclude.
\end{remark}

\begin{remark}
    Thanks to the definition of the family $\mathcal T$, it is possible to link the definition \eqref{eq:quasi_envel} with the classical Dacorogna's formula for the quasiconvex envelope of a functional, see \cite{Dac-book}. Indeed, for a given $(a,X)\in\A(\R^2\times \R^{2\times 2})$ with maximal decomposition $(a,X)=\sum_{j=1}^Jq_j[(a_j,X_j)]$, we can write 
    \begin{equation*}
    A(a,X):=\inf\left\lbrace \fint_{D} \sum_{j=1}^J \bar \psi\Big( \tau_{X_j} \nabla \varphi_j \Big)\,  :\,   \varphi_j\in W^{1,\infty}_0(D,\mathcal A_{q_j})\mbox{ for } j=1,\dots,J  \right\rbrace.
\end{equation*} 
The dependence of the integrand on $a$ is hidden into the choice of the maximal multiplicity for $(a,X).$
\end{remark}

\begin{remark}\label{rmk:dep_A}
    For any $(a,X)\in\A(\R^2\times \R^{2\times 2})$ the value of $A$ is constant in the subset of points $(b,X)\in\A(\R^2\times \R^{2\times 2})$ having the same maximal multiplicities  of $(a,X)$. Indeed, consider   $b\in(\R^2)^Q$ such that $(a,X),(b,X)$ admit maximal decompositions $\sum_{j=1}^J q_j[(a_j,X_j)]$ and $\sum_{j=1}^J q_j[(b_j,X_j)]$ respectively. Then for every $j=1,\dots,J,$ it holds $\tau_{(b_j-a_j)}q_j[a_j+X_j\,\cdot]=q_j[b_j+X_j\,\cdot]$. Thus, every element in $\mathcal T(a,X,D)$ differs from an element in $\mathcal T(b,X,D)$ by translations of elements of its maximal decomposition.  Since $\psi$ depends only on the gradient variable, by \eqref{eq:quasi_envel} we conclude that $A(a,X)=A(b,X).$
\end{remark}

\begin{remark}\label{rmk:split_A}
    For every $q\le Q$, let  $A^q$ be  the functional defined as in \eqref{eq:quasi_envel} on $\mathcal A_{q}$. Consider $(a,X)\in\A(\R^2\times \R^{2\times 2})$ with maximal decomposition  $\sum_{j=1}^J q_j[(a_j,X_j)] $. Since in the definition of $\mathcal T(a,X,D)$ the selections of each map in $\tau_{(a_j+X_j\,\cdot) }W^{1,\infty}_0(U,\mathcal A_{q_j}(\R^2))$  are decoupled and the extension $\bar\psi$ is linear, it holds 
    \begin{equation}\label{eq:A_split}
        A(a,X)=\sum_{j=1}^J A^{q_j}(q_j[(a_j,X_j)]).
    \end{equation}
\end{remark}

\subsection{Basic properties of the integrand $A$}

We are going to use the following interpolation result, in the spirit of \cite[Lemma 2.15]{DeLSpa11}.  

\begin{lemma}\label{lem:interp}
    There exists a constant $C(Q)>0$ with the following property. Let $r>0$,  $\sigma\in(0,1)$ and consider $g\in Lip(\bd D_{(1+\sigma)r},\A)$ and $f\in Lip(\bd D_r,\A)$   with  Lipschitz constant $L_g,L_f$ respectively. Then there exists 
    $h\in Lip( D_{(1+\sigma)r}\setminus D_r,\A)$ with Lipschitz constant $L_h$ such that $ h\llcorner\bd D_{(1+\sigma )r}=g,h\llcorner\bd D_r=f $ and
    \begin{equation}\label{eq:Lip_h}
        L_h\le C \Big(L_f+L_g+ \frac1{\sigma r} \sup_{x\in\bd D_r}\G(f(x),g((1+\sigma)x))\Big).
    \end{equation}

\end{lemma}

\begin{proof}
    By a scaling argument we can consider $r=1$. We denote $\psi= \bm {\xi}\circ  g$ and $\varphi=\bm{\xi}\circ f$, where $\bm \xi:\A\to \R^N$ is  Almgren's bi-Lipschitz embedding  {recalled in Section \ref{sub:Q-val maps}}, and consider the retraction  $\bm \rho:\R^N\to\bm\xi(\A)$.  Denote by $L_\psi,L_\varphi$ the Lipschitz constant of $\psi,\varphi$ respectively. Consider two Lipschitz extensions $\bar \psi,\bar \varphi$ to the whole $\R^2$ of the functions $ \psi, \varphi$ respectively, which can be chosen to have  Lipschitz constants $L_\psi, L_\varphi$.    
    For every $x\in D_{1+\sigma}\setminus D$, we define
    \begin{equation}\label{eq:def_interp}
        \Phi(x)=\tfrac{2\|x\|_\infty-1}{\sigma} \bar \psi(x)+ (1-\tfrac{2\|x\|_\infty-1}{\sigma}) \bar\varphi(x)
    \end{equation}
    and then $h=\bm \xi^{-1}\circ \bm \rho\, \circ \Phi$. It is immediate to check that $h\llcorner\bd D=f, h\llcorner\bd D_{(1+\sigma)}=g$, we only need to bound the Lipschitz constant of $h$. Since   $\bm{\xi}$ is {bi-Lipschitz} and $\bm{\rho}$ is  a retraction, it is enough to bound the Lipschitz constant of $\Phi$.
   
    For  $x\in\R^2$ we  denote $  x_{in} \in \bd D$ and $  x_{out} \in \bd D_{1+\sigma}$ the projection of $x$ on $\bd D$, and on $\bd D_{1+\sigma}$, respectively. For every $x,y\in  D_{1+\sigma}\setminus D$,  we estimate 
    \begin{multline*}
    |\Phi(x) -\Phi(y)| = \Big| \tfrac{2\|x\|_{\infty}-1}\sigma (\bar \psi(x)-\bar \psi(y)) + \tfrac{2\|x\|_{\infty}-2\|y\|_{\infty}}\sigma \bar \psi(y) +  (1-\tfrac{2\|x\|_{\infty}-1}\sigma)  (\bar \varphi(x)-\bar \varphi(y)) - \tfrac {2\|x\|_{\infty}-2\|y\|_{\infty}}\sigma \bar \varphi(y) \Big| \\
    \le \tfrac{2\|x\|_{\infty}-1}\sigma L_\psi|x-y|+(1-\tfrac{2\|x\|_{\infty}-1}\sigma) L_\varphi|x-y|+ \Big| \tfrac{2\|y\|_{\infty}-2\|x\|_{\infty}}\sigma(\bar \psi(y)-\bar \varphi(y))   \Big|       \\
    \le \Big( L_\psi+L_\varphi \Big)|x-y| + \frac{2|x-y|}\sigma  \Big|  \bar \psi(y)-\bar \psi(y_{out}) + \bar \psi(y_{out})- \bar \varphi(y_{in}) + \bar \varphi(y_{in})  -\bar \varphi(y)      \Big|\\
    \le  \Big( L_\psi+L_\varphi \Big)|x-y| + \frac{2|x-y|}\sigma ((L_\psi+L_\varphi) \sqrt 2\,\sigma  + \max_{z\in\bd D}|\psi((1+\sigma)z)-\varphi(z)|)\\
    \le C\Big( L_g+L_f+ \frac1\sigma\max_{z\in\bd D}\G(f(z),g((1+\sigma)z)) \Big)|x-y|,
    \end{multline*}
where in the last line the constant $C>0$ depends on Almgren's {bi-Lipschitz} embedding.

\end{proof}

We now study  the continuity properties of $A$. Crucially, our integrand $A$ turns out to be, in general, only lower semicontinuous and not continuous.
\begin{lemma}\label{lem:A_cont}
    The integrand $A$ defined in \eqref{eq:quasi_envel} is lower semicontinuous on $\A(\R^2\times \R^{2\times 2})$. Moreover, for every   sequence $\{(a^k,X^k)\}_{k\in\N}$ of points in $\A(\R^2\times \R^{2\times 2})$ with the same maximal multiplicities of $(a,X)\in \A(\R^2\times \R^{2\times 2})$ such that $\G((a^k,X^k),(a,X))\to 0$ as $k\to +\infty$, it holds 
    \begin{equation}
        \label{cont}
    \lim_{k\to +\infty} A(a^k,X^k)=A(a,X).
    \end{equation}
    Furthermore, the integrand $A$ satisfies for some {$C=C(Q)>0$} 
    \begin{equation}\label{eq:bound_0}
        A(a,X)\le C(1+\|X\|^2).
    \end{equation}  
\end{lemma}

\begin{proof}
    Fix $(a,X)\in \A(\R^2\times \R^{2\times 2})$ with maximal decomposition $(a,X)=\sum_{j=1}^J q_j[(a_j,X_j)]$. We first observe that for every $f\in Lip(D,\R^2)$ of Lipschitz constant $L_f$, then 
    \begin{equation}\label{eq:bound_0_psi}
    \int_D \psi(\nabla f)\le \int_D\| \bigwedge M(\nabla f) \|\le C\int_D(1+\|\nabla f\|^2)\le C(1+L_f^2),
    \end{equation}
    where $C>0$ is a dimensional constant. 
    In particular,      \eqref{eq:bound_0} follows from this estimate  testing the definition of $A$ with $\sum_{j=1}^J q_j[a_j+X_j\,\cdot]$.

    Consider then a sequence $(a^k,X^k)$ in $\A(\R^2\times \R^{2\times 2})$ converging to $(a,X)$. For  $k$ large enough (depending on $\min_{j\neq i}|(a_j,X_j)-(a_i,X_i)|>0$) there exists a maximal decomposition of $(a^k,X^k)$ of the following form
    \begin{equation}\label{eq:split_an}
    (a^k,X^k)= \sum_{j=1}^{J^k} q^k_j[(a^k_j,X^k_j)],
    \end{equation}
    where $J^k\ge J$ and (up to a permutation) there exist indices $i_0=0,i_1,\dots,i_{J-1},i_J=J^k$ such that  $q^k_{i_{j-1}+1}+\dots+q^k_{i_j}=q_j$  for all $j=1,\dots, J$. 
    
    Set  $\sigma_k:=\G((a,X),(a^k,X^k))\to 0$ as $k\to +\infty$ and 
     denote $u_0:=\sum_{j=1}^J q_j[a_j+X_j\,\cdot]$ and $u_k:=\sum_{j=1}^{J^k} q_j^k[a_j^k+X_j^k\,\cdot]$.     
    In particular we get that $\|X^k\|\le \|X\|+\sigma_k$, and  
    \begin{equation}\label{eq:cont_est0}
        \G(u_0(x), u_k(x))\le \sigma_k, \qquad \forall x\in \bd D.
    \end{equation}
    For every $j=1,\dots ,J$, let $h^k_j\in Lip(  {D_{1+\sigma_k}}\setminus D, \mathcal A_{q_j})$ denote the interpolation produced by Lemma~\ref{lem:interp} between
    \begin{equation}\label{eq:interp_bds}
            q_j[a_j+X_j\,\cdot]\llcorner\bd D_{1+\sigma_k} \quad \text{and}\quad  \sum_{i=i_{j-1}+1}^{i_j} q^k_i [a^k_i+X^k_i\,\cdot]\llcorner\bd D,
    \end{equation}
    and denote $h^k=\sum_{j=1}^J h^k_j$. By \eqref{eq:cont_est0} and the triangular inequality, for $x\in\bd D$  it holds 
    $$\G(u_k(x), u_0((1+\sigma_k )x))\le \sigma_k + \sigma_k \|X\|. $$
    Therefore, denoting $L_k$ the Lipschitz constant of $h^k$, the estimate  \eqref{eq:Lip_h} implies 
    \begin{equation}\label{eq:lip_interp}
        L_k \le  C\Big( 1+ \|X\| \Big).
    \end{equation}

    For $\delta>0$ let   $\varphi^k\in \mathcal T(a^k,X^k,D)$ be such that $A(a^k,X^k)\ge \fint_D\bar\psi(\nabla \varphi^k)-\delta$. 
    By definition of $\mathcal T$, the function $\varphi^k$ admits a decomposition   $\varphi^k=\sum_{i=1}^{J^k}\varphi^k_i, $    with $\varphi^k_i\in Lip(D,\mathcal A_{q^k_i})$. 
    In particular, by \eqref{eq:split_an} the function  $\varphi^k$ admits a decomposition
    \[
    \varphi^k= \sum_{j=1}^J  \tilde \varphi^k_j
    \quad \mbox{ where }  \quad \sum_{i=i_{j-1}+1}^{i_j}\varphi^k_i =:\tilde \varphi^k_j\in Lip(D,\mathcal A_{q_j}) \mbox{ for every $j=1,\dots, J$}.\]
    We define  the following function on $D_{1+\sigma_k}$
    \[
    \eta^k=\sum_{j=1}^J ( h^k_j\chi_{D_{1+\sigma_k}\setminus D} + \tilde \varphi^k_j \chi_{D}).
    \]
    By  \eqref{eq:lip_interp}, the function $\eta^k$ is Lipschitz  and satisfies $\eta^k\llcorner\bd D_{1+\sigma_k}=u_0\llcorner\bd D_{1+\sigma_k}$, thus  $\eta^k\in\mathcal T(a,X,D_{1+\sigma_k}).$ 
    By the choice of $\varphi^k$,  \eqref{eq:bound_0_psi} and \eqref{eq:lip_interp}, we can estimate 
 \begin{multline}\label{eq:comp_lsc_A}
        A(a^k,X^k)\ge   \fint_D \bar\psi(\nabla \varphi^k)-\delta = |D_{1+\sigma_k}|\fint_{D_{1+\sigma_k}}\bar \psi(\nabla \eta^k) - \int_{D_{1+\sigma_k}\setminus D}\bar \psi (\nabla h^k)-\delta\\
        \ge \fint_{D_{1+\sigma_k}}\bar \psi(\nabla \eta^k) - C|D_{1+\sigma_k}\setminus D| \, (1+L_k^2)  -\delta     \ge \fint_{D_{1+\sigma_k}}\bar \psi(\nabla \eta^k) - C\sigma_k(1+\|X\|^2) -\delta.       
    \end{multline}
    Since $\eta^k\in \mathcal T(a,X, D_{1+\sigma_k})$, by the definition of $A$  and Remark~\ref{rmk:nidepend}, we get 
    \[
    A(a^k,X^k)\ge A(a,X) -C\sigma_k(1+\|X\|^2)- \delta .
    \]
    We conclude the lower semicontinuity of $A$ by letting $k\to +\infty$ and using that $\delta$ is arbitrary.

    The continuity of $A$ on the subset of $\A(\R^2\times \R^{2\times 2})$ consisting of points with the same maximal multiplicities of  $(a,X)$ follows essentially the same argument as before. The only difference is that, since the maximal multiplicities of $(a,X)$ and $(a^k,X^k)$ now coincide, we can interpolate between $u_0$ on $D$ and $u_k$ on $D_{1+\sigma_k}$ by means of a map $h^k$, which  admits a decomposition into $J$ maps in $\mathcal A_{q_1},\dots,\mathcal A_{q_J}$ respectively. This allows  to reproduce \eqref{eq:comp_lsc_A}, starting from $A(a,X)$, to establish the reverse inequality and thus  the desired \eqref{cont}.

\end{proof}

\begin{remark}\label{rmk:A_not_C0}
    We observe that the continuity of the integrand $A$ is not expected in general. For instance $A(Q[(0,0)])$ can be in general strictly smaller than the constant value $A(a,0)$ evaluated on all $(a,0)\in \A(\R^2\times \R^{2\times 2})$ with the same maximal multiplicities $(q_1,\dots,q_J)$, if $J>1$. Indeed the family of test functions in \eqref{eq:quasi_envel} for $A(Q[0,0])$ is strictly larger than the one for $A(a,0)$, up to a translation. And it is easy to approximate $(0,0)\in \A(\R^2\times \R^{2\times 2})$ with a sequence $(a^k,0)$ as above.
    
    This obstruction also occurs if one tries to prove upper semicontinuity of $A$, following the proof of Lemma~\ref{lem:A_cont}, but reversing the roles of the maps in \eqref{eq:interp_bds}. More precisely, consider the case where the maximal decomposition of $(a,X),(a^k,X^k)$ satisfies \eqref{eq:split_an} and there exists $j\in\{1,\dots, J\}$ such that $i_j\neq 1$. Assume for simplicity that  $i_{1}>1$. 
    Then a test function $\varphi\in \mathcal T(a,X,D)$ can be decomposed as $J$ maps  $\varphi_1,\dots,\varphi_J$, each  valued  in  $\mathcal A_{q_j}$.  But the map $\varphi_1 $  in general does not admit a Lipschitz decomposition into  $i_1$ Lipschitz maps, each in $\mathcal A_{q_1^k},\dots,\mathcal A_{q_{i_1}^k}$.  In particular, the resulting map $\eta^k$ does not belong to $\mathcal T(a^k,X^k,D_{1+\sigma_k})$ since it does not admit the correct decomposition in $D_{1+\sigma_k}$. Therefore we cannot conclude the reverse inequality of \eqref{eq:comp_lsc_A}.  
\end{remark}

\subsection{Lower semicontinuity of the energy associated to $A$}
 
The main goal of the section is showing the following result.

\begin{theorem}\label{prop:quasi_envel}
    The lower semicontinuous  $Q$-integrand $A$ on $\A(\R^2\times \R^{2\times 2})$ defined in \eqref{eq:quasi_envel} gives rise to the following  energy  on $W^{1,p}(D, \A(\R^2))$  for every $p\ge 2$ 
    \[
    \mathcal F(u) = \fint_D A(u(x),\nabla u(x))\,  dx \qquad \forall u\in W^{1,p}(D,\A(\R^2)),
    \]
    which is weakly  lower semicontinuous  in $W^{1,p}(D,\A(\R^2)).$
\end{theorem}

{We refer to Section \ref{sub:Q-val maps} for the relevant  notions regarding  Sobolev $Q$-valued functions.}
A crucial step of the proof of Theorem~\ref{prop:quasi_envel} is to show that $A$ satisfies the following property.

\begin{assumptionA}\label{hp:A}
    Let $F$ be a $Q$-integrand on $\A(\R^2\times \R^{2\times 2})$. We say that $F$ satisfies Property \ref{hp:A} if the following holds. Let $(a,X)\in\A(\R^2\times \R^{2\times 2})$ with maximal decomposition $\sum_{j=1}^J q_j[(a_j,X_j)]$ with $a_j\neq a_{j'}$ for $j\neq j'$. In particular, there exists  $r>0$ such that the graph  of  $x\mapsto \sum_{j=1}^J q_j[a_j+X_j\,x]$ coincides in $D_r$ with  the union of $J$ affine planes with multiplicities, at positive distance one from the other in $D_r$.  
    Then, for every $f=\sum_{j=1}^J f_j\in\mathcal T(a,X,D_r)$ such that for every $x\in D_r$ the sets spt$(f_j(x))$ are  pairwise disjoint  for different indices $j$, it holds
    \begin{equation}\tag{A} \label{eq:weak-qcvx2}
        F(a,X)\le \fint_{D_r} F(f,\nabla f).
    \end{equation} 
\end{assumptionA}

In the proof of Theorem~\ref{prop:quasi_envel} we will actually show that the integrand $A$ defined in \eqref{eq:quasi_envel} enjoys the following slightly stronger property:

\begin{assumptionB}
\label{hp:B}
    Let $A$ be the $Q$-integrand on $\A(\R^2\times \R^{2\times 2})$ defined in \eqref{eq:quasi_envel}. We say that $A$ satisfies Property \ref{hp:B} if the following holds. For every  $(a,X)\in\R^2\times \R^{2\times 2}$ and for every $q\le Q$ and 
    function $f\in\mathcal T(q[(a,X)],D)$,   it holds  
\begin{equation}\tag{B}\label{eq:weak-qcvx}
    A^q(q[(a,X)])\le \fint_{D} A^q(f,\nabla f),
\end{equation}   
where the $q$-integrand $A^q$ is defined in Remark \ref{rmk:split_A}.
\end{assumptionB}

{We note that, thanks to Remark~\ref{rmk:split_A}, once
Property~\ref{hp:B} is established, it follows
that the integrand \(A\) satisfies Property~\ref{hp:A}.} 
Indeed, assuming Property~\ref{hp:B} and  {arguing} as in Remark~\ref{rmk:nidepend},  for every $\lambda>0$ and  $q\le Q$ it holds  
\[
    A^q(q[(a,X)])\le \fint_{ D_\lambda} A^q(f,\nabla f),\quad \forall f\in\mathcal T(q[(a,X)], D_\lambda).
\]
Then, since  for every $x\in D_r$,  spt$(f_j(x))$ are  disjoint for different indices $j$, by Remark~\ref{rmk:split_A}  it holds  $A(f(x),\nabla f(x))=\sum_{j=1}^J A^{q_j}(f_j(x),\nabla f_j(x)) $ and the same holds for $A(a,X)$. Thus, applying \eqref{eq:weak-qcvx} to each $A^{q_j}$ we get \eqref{eq:weak-qcvx2}.

\begin{proof}[Proof of Theorem~\ref{prop:quasi_envel}.]
We divide the proof in two steps.

\noindent\textit{Step 1:} 
We show that $\mathcal F$ is weakly  lower semicontinuous  in $W^{1,p}(D,\A(\R^2))$,  assuming  that  the integrand $A$ satisfies Property~\ref{hp:A}. 
The proof of this step   essentially follows the proof of \cite[Theorem 0.2]{DeLFocSpa}, specifically the part showing the sufficiency of quasiconvexity for lower semicontinuity of the energy, and is based on the blow-up procedure introduced in \cite{FoMu}. We recall briefly the main steps of the proof, highlighting the main differences with the reference \cite{DeLFocSpa}.

In order to show the lower semicontinuity  of 
the functional $\mathcal F$, we consider $p\in[2,+\infty)$ and  a sequence $v^k$ converging  to $u$ weakly in $W^{1,p}(U,\A)$ ({in the sense recalled in Section \ref{sub:Q-val maps}}). Note that, by \eqref{eq:bound_0}, the integrand $A$ satisfies $A(a,X)\le C(1+\|X\|^p)$ for every $p\ge 2$.
By \cite[Lemma~1.5]{DeLFocSpa}, one can replace $v^k$ by $u^k$ such that: the functions $u^k$ are in $W^{1,\infty}(U,\A)$ and  are  converging to $u$ weakly in $W^{1,p}$, there exists $\lim_k \mathcal F(u^k)$
and the functions $u^k$ are equi-integrable, meaning that there exists a  function $\varphi:[0,+\infty]\to [0,+\infty]$ superlinear at infinity such that 
\[
\sup_k\int_{U} \varphi(|\nabla u^k|^p)<+\infty.
\]
To conclude, it is sufficient to show that $\mathcal F(u)\le \lim_k \mathcal F(u^k),$
which is, in turn, implied by showing that 
\begin{equation}\label{eq:density}
    A(u(x),\nabla u(x))\le \frac{d \mu}{d \mathscr L^2}(x)
\end{equation}
holds on a set of full measure $\tilde U$, where $\mu$ is (up to subsequences) the weak-$*$ limit of the measures $A(u^k,\nabla u^k)\mathscr L^2\llcorner U$, see \cite{FoMu}. The set $\tilde U$ is chosen in the reference \cite{DeLFocSpa}   so that some density conditions for $u^k,u$ hold on it. In particular, $u$ is differentiable at every point $x\in\tilde U$. 
The proof of \eqref{eq:density} follows from two claims, as in \cite[Theorem 0.2]{DeLFocSpa}.

\noindent \textbf{Claim 1:}
Let $x_0\in\tilde U$ and $(u(x_0),\nabla u(x_0))$ with maximal decomposition
$ \sum_{j=1}^Jq_j [(a_j,X_j)]$. We assume $J>1$, as otherwise the arguments below simplify. Denote $\delta=\min_{j\neq j'}|a_j-a_{j'}|>0$ and choose  $0<\alpha\le \min\{ 1, \delta/(8Q|\nabla u|(x_0)\,) \}$, with $\alpha=1$ if $|\nabla u|(x_0)=0$. 
Then, there exist $\rho_k\to  0$ and $\{w^k\}_{k\in\N}\subseteq Lip(D(x_0,\rho_k),\A)$ such that:
\begin{itemize}
\item[(a)]  $w^k$ admits a decomposition  $w^k=\sum_{j=1}^J[w_j^k]$ with $w^k_j \in Lip(D(x_0,\rho_k),\mathcal A_{q_j})$, and for every $x\in D(x_0,\rho_k)$ it holds 
$\G(w^k(x),u(x_0)) \le2 \alpha\max\{|\nabla u|(x_0),\delta/(8Q)\}\, \rho_k$ and also  $  \G(w^k(x),u(x_0))^2=\sum_{j=1}^J\G(w^k_j(x),q_j[a_j])^2;$
\item[(b)] $\fint_{D(x_0,\alpha \rho_k)}\G^p(w^{k},T_{x_0}u)=o(\rho_k^p)$;
\item[(c)] $\lim_{k\to+\infty}\fint_{D(x_0,\alpha\rho_k)}   A\big(w^k,\nabla w^k\big)=\frac{d\mu}{d\mathscr L^2}(x_0)$.
\end{itemize}
Note the difference from  \cite[Theorem 0.2]{DeLFocSpa} in item (c).
We recall the main steps of the proof.

As in \cite[Theorem 0.2]{DeLFocSpa}, we can find radii $\rho_k$ satisfying the following conditions:
\begin{gather}
\sup_k \fint_{D(x_0,\rho_k)}\varphi(|\nabla u^{k}|^p)<+\infty,\quad 
\fint_{D(x_0,\alpha\rho_k)} A\big(u^k,\nabla u^k\big)\to\frac{d\mu}{d\mathscr L^2}(x_0),\nonumber\\
\fint_{D(x_0,\rho_k)}\G^p(u^{k},u)=o(\rho_k^p)\quad
\text{and}
\quad\fint_{D(x_0,\rho_k)}\G^p(u^{k},T_{x_0}u)=o(\rho_k^p).\label{e:rad3}
\end{gather}
Set $r_k=2\alpha\,\max\{|\nabla u|(x_0),\delta/(8Q)\}\rho_k$ and consider the retraction  $\vartheta_k:\A\to \overline {B}_{r_k}(u(x_0))\subset \A$
constructed in \cite[Lemma 3.7]{DeLSpa11}.
We define $w^k:=\vartheta_k\circ u^k$, and check that the claim holds.

Indeed, (a) follows  since $\vartheta_k$ takes values in $\overline {B}_{r_k}(u(x_0))\subset \A$.
As for (b), the choice of $r_k$ implies that
$\vartheta_k\circ T_{x_0}u=T_{x_0}u$ on $D(x_0,\alpha\rho_k)$ since 
\begin{equation}\label{e:rk}
\G(T_{x_0}u(x),u(x_0))\leq |\nabla u(x_0)|\,|x-x_0|\leq
|\nabla u(x_0)|\,\alpha\rho_k\le\frac{r_k}{2}.
\end{equation}
{We can then argue} exactly as in \cite[Theorem 0.2]{DeLFocSpa} to show that  $A_k:=\big\{w^k\neq u^k\big\}$ 
satisfies $|A_k|=o(\rho_k^2).$ 
We thus show item (c), since   by the definition of $A_k$ it holds 
\begin{multline*}
\left|\fint_{D(x_0,\alpha\rho_k)}A\left(u^k,\nabla u^k\right)-\fint_{D(x_0,\alpha\rho_k)} A\left(w^k,Dw^k\right)\right|\\
\leq\frac 1{\alpha^2\rho_k^{2}}\int_{A_k}\Big( A\left(u^k,\nabla u^k\right)+A\left(w^k,Dw^k\right)\Big)
\leq \frac{C}{\rho_k^{2}}\int_{A_k}\big(1+|Dw^k|^p+|\nabla u^k|^p\big).  
\end{multline*}
By the previous estimate and  the equi-integrability of $\nabla u^k$, $\nabla w^k$ we conclude  using  \cite[Lemma~1.6]{DeLFocSpa}.

Using Claim 1, one can ``blow-up'' the functions 
$w^k$, as shown in the following claim.

\noindent \textbf{Claim 2:} For every $\gamma>0$, there exist $\{z^k\}_{k\in\N}\subseteq Lip(D_\alpha,\A)$, 
with the following properties. For each $k\in\N$, $z^k$ admits a  decomposition $z^k=\sum_{j=1}^J z^k_j$ in Lipschitz functions $z^k_j\in Lip(D_\alpha,\mathcal A_{q_j})$ such that, for every $x\in D_\alpha$, spt$(z^k_j(x))$ are   disjoint  for different indices $j$.  Moreover,  $z^k\in\mathcal T(u(x_0),\nabla u(x_0),D_\alpha)$  and
\begin{equation}\label{eq:conv_zk}
\limsup_{k\to+\infty}\fint_{D_\alpha} 
  A\big(z^k,\nabla  z^k\big)\leq \frac{d\mu}{d\mathscr L^2}(x_0)+\gamma.    
\end{equation}

To show Claim 2, as  in \cite[Theorem 0.2]{DeLFocSpa}, one starts with  the functions  $w^k$ of Claim 1. Since they have the same multiplicity of $u$  at $x_0$, we can blow-up.
Let $\zeta^k:=\sum_{j=1}^J\zeta^k_j$ with the maps $\zeta^k_j\in Lip(D_\alpha,\mathcal A_{q_j})$ defined by
$\zeta^k_j:=\tau_{a_j}\rho_k^{-1}\tau_{-a_j}w^k_j(x_0+\rho_k\cdot)$.

Note that for every $j=1,\dots,J$ and  $x\in D_\alpha$, it holds 
\begin{equation}\label{eq:dist_spts}
\G(\zeta^k_j(x),q_j[a_j])\le \tfrac1{\rho_k}\G\big(w^k_j(x_0+\rho_kx) ,q_j[a_j]\big)
\le 2\alpha\max\{|\nabla u|(x_0),\delta/(8Q)\}\le  \frac \delta {4}.    
\end{equation}
In particular,  for every $x\in D_\alpha$,   spt$(\zeta^k_j(x))$ are  disjoint  for different indices $j$.  By a change of variable, one also gets  $\zeta^k_j\to q_j[ a_j+X_j\,\cdot]$ in $L^{p}(D_\alpha,\mathcal A_{q_j})$ and 
\[
\fint_{D_\alpha} A\big(\zeta^k ,\nabla \zeta^k\big)=\fint_{D(x_0,\alpha\rho_k)}  A\Big(\sum_{j=1}^J \tau_{a_j}\rho_k^{-1}\tau_{-a_j}[w^k_j],\nabla w^k \Big).
\]
For  every  $x\in D(x_0,\alpha\rho_k)$, the maximal multiplicities of  $\tau_{a_j}\rho_k^{-1}\tau_{-a_j}[w^k_j( x)]$ and $w^k_j(x)$ coincide. Since spt$(w^k_j(x))$ are disjoint for different $j$, also the maximal multiplicities of $w^k(x)$ and $\sum_{j=1}^J \tau_{a_j}\rho_k^{-1}\tau_{-a_j}[w^k_j(x)]$ coincide for every $x\in D_\alpha$. By Remark~\ref{rmk:dep_A} we   get 
\[
\fint_{D(x_0,\alpha\rho_k)}  A\Big(\sum_{j=1}^J \tau_{a_j}\rho_k^{-1}\tau_{-a_j}[w^k_j],\nabla w^k \Big)=\fint_{D(x_0,\alpha\rho_k)}  A(w^k,\nabla w^k).
\]
Therefore, by item (c) in the previous claim we conclude 
\begin{equation*}
\lim_{k\to+\infty}\fint_{D_\alpha}  A\big(\zeta^k ,\nabla \zeta^k\big)
=\frac{d\mu}{d\mathscr L^2}(x_0).
\end{equation*}
The rest of the proof of Claim 2 follows as in \cite[Theorem 0.2]{DeLFocSpa}, and amounts to employ Lemma~\ref{lem:interp} to match the boundary data $T_{x_0} u\llcorner\bd D$. The only difference is that we first define
$z^k_j\in Lip(D_\alpha,\mathcal A_{q_j})$ for $j=1,\dots,J$ interpolating between  $\zeta^k_j\llcorner D_{r\alpha}$ and $q_j[a_j+X_j\,\cdot]\llcorner \bd D_\alpha$,  and then  $z^k=\sum_{j=1}^J z^k_j\in\mathcal T(u(x_0),\nabla u(x_0),D_\alpha)$. Here, the parameter $r\in(0,1)$ is chosen as in \cite[Theorem 0.2]{DeLFocSpa}. We now check that  for every $x\in D_\alpha$, spt$(z^k_j(x))$ are  disjoint   for different indices $j$. If  $x\in D_{r\alpha}$, then $z^k_j(x)=\zeta^k_j(x)$ and spt$(\zeta^k_j(x))$ are  disjoint   for different indices $j$. Consider now $x\in  D_\alpha\setminus D_{r\alpha}$. By definition of the interpolation $\Phi$ in \eqref{eq:def_interp}, we can estimate for every $P\in \mathcal A_{q_j}$ 
\begin{multline*}
\G(h^k_j(x),P)\le Lip(\bm\xi^{-1})|\bm\rho\circ \Phi(x)-\bm\xi(P)|=Lip(\bm\xi^{-1})|\bm\rho\circ \Phi(x)-\bm\rho\circ \bm\xi(P)|  \le Lip(\bm\xi^{-1})| \Phi(x)-\bm\xi(P)|\\
\le  Lip(\bm\xi^{-1}) Lip(\bm\xi)\Big(  \tfrac{\alpha-2\|x\|_\infty}{(1-r)\alpha} \max_{y\in D_{\alpha}}\G(\zeta^k_j(y),P) +( 1-\tfrac{\alpha-2\|x\|_\infty}{(1-r)\alpha})  \max_{y\in D_\alpha}\G(q_j[a_j+X_j\,y],P  ) \Big)\\
\le  Lip(\bm\xi^{-1}) Lip(\bm\xi)\max\Big\{  \max_{y\in D_\alpha}\G(q_j[a_j+X_j\,y],P  ) , \max_{y\in D_{\alpha}}\G(\zeta^k_j(y),P)  \Big\}.
\end{multline*}
Considering  $P=q_j[a_j]$ and up to choosing a smaller $\alpha$, we  conclude that 
\[
\G(h^k_j(x),q_j[a_j])\le C(Q) 2\alpha\max\{|\nabla u|(x_0) , \tfrac{\delta}{8Q} \} \le \tfrac\delta4,
\]
which implies that the supports are pairwise disjoint.

By Claim 2, we can test  \eqref{eq:weak-qcvx2} with every $z^k$, deducing  
$$
A\big(u(x_0),\nabla  u(x_0)\big)\leq\limsup_{k\to+\infty}\fint_{D_\alpha}A \big(z^k,\nabla z^k\big)\leq\frac{d \mu}{d\mathscr L^2}(x_0)+\gamma,
$$
which implies \eqref{eq:density} by letting $\gamma\to 0$ and concludes the proof.

\vspace{2ex}

\noindent\textit{Step 2:} We show that the integrand $A$ defined in \eqref{eq:quasi_envel} satisfies Property~\ref{hp:A}. {This is obtained by proving that the integrand $A$ satisfies}   Property~\ref{hp:B}, that is \eqref{eq:weak-qcvx} holds for every $q\le Q$ and every    function $f\in\mathcal T(q[(a,X)],D)$. Since the argument below does not depend on the particular value of $q$ we choose, we fix $q=Q$. 
We fix $(a,X)\in\R^2\times \R^{2\times 2}$ and,   with an abuse of notation,  we denote $(a,X)=Q[(a,X)]\in\A$.

Fix $f\in\mathcal T(a,X,D)$ and $\delta >0$. 
We first note that the energy of $f$ can be approximated by the energy of almost piecewise affine functions. In fact, denoting with $L$ the Lipschitz constant of $f$, by Lemma~\ref{lem:polyapprox}, we find an $L'(L,Q)-$Lipschitz function $g\in\mathcal T(a,X,D)$ for which there exist $K\in\N$ pairwise disjoint cubes $D^\ell \subseteq D$ and points $(a^\ell,X^\ell)\in\A(\R^2\times \R^{2\times 2})$ such that 
$$|D\setminus \bigcup_{\ell=1}^{K}D^\ell|\le \delta,\quad g\llcorner D^\ell=:g^\ell \text{ is affine   and }   g^\ell\in\mathcal T(a^\ell,X^\ell,D^\ell),
$$
where $X^\ell=\nabla g^\ell$, and $g^\ell$ is  the Taylor expansion of $f$ at a differentiability point for $f$ in $D^\ell$. Moreover   
$$
\int_D A(f,\nabla f)\ge  \int_D A(g,\nabla g)-\delta.
$$
We are thus left to show \eqref{eq:weak-qcvx} for the function $g$.  

We denote $D^\ell=x^\ell+r D$ for $x^\ell\in D$, and recall  that   $r\in (0, \delta )$. In particular it holds 
\begin{equation}\label{eq:est_A_0}
    \int_D A(g,\nabla g)
    \ge \sum_{\ell=1}^{K}\int_{D^\ell} A(g,\nabla g)
    = \sum_{\ell=1}^{K}\int_{D^\ell} A(g,X^\ell).
\end{equation}
Since $g^\ell$ is affine,  for every $x\in \R^2$, the point $(g^\ell(x),\nabla g^\ell(x))$ has the same maximal multiplicities of    $(a^\ell,X^\ell)$. Thanks to Remark~\ref{rmk:dep_A} it thus holds $A(g^\ell(x),X^\ell)=A(a^\ell,X^\ell)$ and \eqref{eq:est_A_0} becomes
\begin{equation}\label{eq:est_A_1}
 \int_D A(g,\nabla g)
    \ge \sum_{\ell=1}^{K} |D^\ell|A(a^\ell,X^\ell).
\end{equation}
For every $\ell=1,\dots, K$ we  consider $\varphi^\ell\in\mathcal T(a^\ell,X^\ell,D^\ell)$  such that 
\begin{equation}\label{eq:est_A_2}
A(a^\ell,X^\ell)\ge \fint_{D^\ell} \bar\psi(\nabla \varphi^\ell) - \delta. 
\end{equation}
Note that  $g\llcorner\bd D^\ell=\varphi^\ell\llcorner \bd D^\ell$. Therefore,  the following $Q$-valued map is in $Lip(D,\A)$:
\[
 \varphi :=g\,\chi_{D\setminus \bigcup_{\ell=1}^{K}D^\ell} +
\sum_{\ell=1}^{K} \varphi^\ell \,\chi_{D^\ell}. 
\]
By \eqref{eq:est_A_1} and \eqref{eq:est_A_2} we get
\[
\int_D A(g,\nabla g)\ge\sum_{\ell=1}^{K}\Big(\int_{D^\ell} \bar\psi (\nabla \varphi^\ell) -\delta|D^\ell|\Big) \ge \int_D \bar\psi(\nabla  \varphi) - C\delta,
\]
where the positive constant $C$ depends on $L'$, and thus on $f$.  Since $(a,X)=Q[(a,X)]$ and $\varphi\llcorner\bd D = Q[a + X\,\cdot]\llcorner\bd D$, then   $ \varphi\in\mathcal T(a,X,D)$. Hence we deduce from the previous inequality and the definition of $A$     \eqref{eq:quasi_envel}, that 
\[
 \int_D A(g,\nabla g)\ge A(a,X)- C\delta.
\]
Since $\delta$ was arbitrary,  we conclude \eqref{eq:weak-qcvx} and thus   the proof.

\end{proof}

An interesting consequence of the previous result is the following characterization. We consider a bounded, open set $U\subseteq \R^2$.

\begin{proposition}\label{prop:iff}
    Assume that a lower semicontinuous integrand $F$  on $\A(\R^2\times \R^{2\times 2})$ satisfies
    \begin{equation}\label{eq:dep_F}
    \begin{split}
        &F(a,X)=F(b,X), \\
        &\mbox{ for all    $(a,X),(b,X)\in\A(\R^2\times \R^{2\times 2})$ with the same maximal multiplicities.}    
    \end{split}
    \end{equation}

    Then, if the associated energy $\mathcal F(f)=\int_U F(f,\nabla f)$
    is weakly$-*$ lower semicontinuous in $W^{1,\infty}(U,\A)$, the functional $F$ satisfies  Property~\ref{hp:A}.

    Conversely, if $F$ satisfies Property~\ref{hp:A} and for some  $p\ge 2,$
    \[F(a,X)\le C(1+\|X\|^p),\]
    then the associated energy $\mathcal F$ is    weakly   lower semicontinuous in $W^{1,p}(U,\A)$.
\end{proposition}

\begin{proof}\textit{Weak}$-*$ \textit{lower semicontinuity implies Property~\ref{hp:A}.} We assume that the functional $\mathcal F$  is weakly$-*$ lower semicontinuous in $W^{1,\infty}(U,\A)$. 
We fix  $(a,X)\in \A(\R^2\times \R^{2\times 2})$ with a maximal decomposition $\sum_{j=1}^Jq_j[(a_j,X_j)]$ and $a_j\neq a_i$ for $j\neq i$.
Let $x_0\in U$ and define the affine function $u(x):=\sum_{j=1}^Jq_j[a_j + X_j\,(x-x_0)]=\sum_{j=1}^Ju_j(x)$.  We can assume without loss of generality that $x_0=0\in U$. Let $D_r\subseteq U$ for some $r>0$. Note that for $r$ small we can ensure that the graph of $u$ over $D_r$ is the union of $J$ affine planes with multiplicities, at positive distance one from the other. Consider a Lipschitz test function   $w =\sum_{j=1}^J w_j\in \mathcal T(a,X,D_r)$ such that for every $x\in D_r$ the sets spt$(w_j(x))$ are  pairwise disjoint  for different indices $j$, as in the statement of Property~\ref{hp:A}.

Set $z_{j}(y):=\tau_{-(a_j+X_j\,y)}w_j(y)$, so that $z_j\llcorner{\bd D_r}\equiv q_j[0]$, and extend $z_j$ as an $r$-periodic map to the whole $\R^2$. 
We consider $v^k_j(y)=\tau_{a_j+X_j\,y}\, ( k^{-1}z_j(ky))$ and $v^k=\sum_{j=1}^J v^k_j$ for $k\in\N$. 
We replace  $v^k_j$ by $u_j$ outside $D_r$.

Note that $v^k_j\rightharpoonup u_j$ weakly$-*$ in $W^{1,\infty}(D_r,\mathcal A_{q_j})$ as uniformly Lipschitz and converging in $L^\infty$. 
Thus, by the lower semicontinuity of $\mathcal F$ we get 
\begin{equation}\label{eq:char-lsc_0}
    \int_{D_r}F(u,\nabla u)\le \liminf_{k\to+\infty} \int_{D_r} F(v^k,  \nabla v^k ).
\end{equation}
Since $u$ is affine,  the maximal multiplicities of $(u,\nabla u)=(u,X)$ are the same as those of $(a,X)$. By \eqref{eq:dep_F} and \eqref{eq:char-lsc_0} we   have 
\begin{equation}\label{eq:char-lsc_1}
|D_r|F(a,X) = \int_{D_r}F(u,\nabla u)\le \liminf_k \int_{D_r} F(v^k, \nabla v^k ).
\end{equation}
We now estimate the right-hand side in \eqref{eq:char-lsc_1}. 
First, for $k$ sufficiently large and for every $x\in   D_r$, spt$(v^k_j(x))$ are pairwise disjoint for different indices $j$. Then, by definition  the maximal multiplicities of $v^k_j(x/k)$ and $w_j(x)$ coincide for all $x\in D_r.$
Since $\nabla v^k$ is $\frac rk$-periodic in $D_r$ and so are the maximal multiplicities of $v^k$, by  periodicity and  a change of variable we get 
\begin{equation}\label{eq:periodic}
\fint_{D_r} F(v^k, \nabla v^k) = \fint_{ D_{r/k}}  F(v^k, \nabla v^k)\\=\fint_{ D_{r/k}}  F(w(k\,\cdot), \nabla v^k)
=\fint_{D_r}   F\Big(w, \sum_{j=1}^{J}  \tau_{X_j}\nabla z_j  \Big),
\end{equation}
where in the second equality we used \eqref{eq:dep_F}. 
Since   $\tau_{X_j}\nabla z_j = \nabla w_j$ by construction, we  conclude  
\begin{equation*} 
	\int_{D_r} F( v^k, \nabla v^k )=\int_{D_r}   F( w ,  \nabla  w),
\end{equation*}
which, paired with \eqref{eq:char-lsc_1}, yields the conclusion.

\textit{Property~\ref{hp:A} implies weak lower semicontinuity.} 
Assume that the integrand   $F$  satisfies Property~\ref{hp:A}. The result follows from Step 1 in the proof of Proposition~\ref{prop:quasi_envel} once we show that for every sequence $\{(a^k,X^k)\}_{k\in\N}\subseteq\A(\R^2\times \R^{2\times 2})$ with the same maximal multiplicities of $(a,X)$, if $\G((a^k,X^k),(a,X))\to 0$ then  $F(a^k,X^k)\to F(a,X)$. This convergence was used in Step 1 in the proof of Proposition~\ref{prop:quasi_envel} to deduce \eqref{eq:density}.  

Let $\sum_{j=1}^Jq_j[(a_j,X_j)]$, $\sum_{j=1}^Jq_j[(a_j^k,X_j^k)]$ be a maximal decomposition for $(a,X),(a^k,X^k)$ respectively, and let $u_0,u^k$ be the affine functions defined by $\sum_{j=1}^Jq_j[a_j + X_j\,\cdot ]$, $\sum_{j=1}^Jq_j[a_j^k + X_j^k\,\cdot ]$ respectively.
Note that, since \eqref{eq:density} needs to be checked at differentiability points of $u$, we can assume that $a_j\neq a_i$ for $j\neq i$.

By lower semicontinuity of $F$, it suffices to show that  
$$
F(a,X)\ge \limsup_{k\to+\infty}F(a^k,X^k).
$$
To do so, we {argue} as in Lemma~\ref{lem:A_cont}. 
Note that, for $k$ sufficiently large, there exists $r>0$ such that, for each $j$, the graph over $D_{2r}$ of the affine map $x\mapsto a_j^k+X_j^k x$ is contained in a tubular neighborhood of the graph of $x\mapsto a_j+X_j x$. Moreover, these neighborhoods are pairwise disjoint for distinct indices $j$.

Let $\sigma_k=\G((a,X),(a^k,X^k))$. We   build an interpolating function $h^k$ between $u_0\llcorner\bd D_{r}$ and $u^k\llcorner\bd D_{(1+\sigma_k)r}$ by Lemma~\ref{lem:interp} with the following properties: it  admits a decomposition $h^k=\sum_{j=1}^J h^k_j$ matching the decomposition  of   $u_0,u^k$;  the  graph   of each $h^k_j$ is contained in the tubular neighborhood of   the graph of $x\mapsto a_j+X_j x$ and they are thus pairwise disjoint for different $j$.  
{Arguing} as for \eqref{eq:lip_interp}, the Lipschitz constant   of $h^k$ is uniformly bounded as $k\to +\infty$. We define on $D_{(1+\sigma_k)r}$ the map $\varphi^k=u_0\chi_{D_r}+h^k \chi_{D_{(1+\sigma_k)r}\setminus D_r}$. Since $\varphi^k\in\mathcal T(a^k,X^k,D_{(1+\sigma_k)r})$, using \eqref{eq:weak-qcvx2} we conclude
\[
    F(a^k,X^k)\le \fint_{D_{(1+\sigma_k)r}}F(\varphi^k,\nabla \varphi^k)\leq  F(a,X) +  C\sigma_k  
\]
for a positive constant $C$. Sending $k\to+\infty$ we conclude.

\end{proof}

We conclude this section with the following lemma which relates  the quasiconvexity introduced in \cite{DeLFocSpa} and Property~\ref{hp:A} for   integrands satisfying \eqref{eq:dep_F}.

\begin{lemma}\label{lem:iff}
    Consider a $Q$-integrand $F$ on $\A(\R^2\times \R^{2\times 2})$ satisfying \eqref{eq:dep_F}. If $F$ is lower semicontinuous,  quasiconvexity of~$F$ implies  Property \ref{hp:A}. If $F$ is continuous, quasiconvexity and Property \ref{hp:A} {are equivalent}. 
\end{lemma}

\begin{proof}
    Let  $(a,X)\in\A$ and $(a,X)=\sum_{j=1}^J q_j[(a_j,X_j)]$ with $a_j\neq a_i$ for $j\neq i$.
We first note the following. Consider $(b,X)\in\A$  admitting a decomposition of the form
\[
    (b,X )= \sum_{j=1}^{J'} q'_j[(b_j,X_j)], \quad b_j\neq b_i\mbox{ for } j\neq i,
\]
where $J'\ge J$ and (up to a permutation) there {exist} indices $i_0=0,i_1,\dots,i_{J-1},i_J=J'$ such that  $q'_{i_{j-1}+1}+\dots+q'_{i_j}=q_j$  for all $j=1,\dots, J$. Then, if $F$ is lower semicontinuous it holds  $F(a,X)\le F(b,X)$, while if $F$ is continuous 
 it holds $F(a,X)=F(b,X)$. Indeed, let  $\{(a^k,X)\}_{k\in\N}\subseteq \A$ be a sequence  with the same maximal multiplicities of $(b,X)$ and converging to $(a,X)$. By semicontinuity of $F$ and \eqref{eq:dep_F} we conclude that $F(a,X)\le \liminf_k F(a^k,X)=F(b,X)$. If $F$ is continuous we have instead $F(a,X)=F(b,X).$
Consider then  $w\in Lip(D_r,\A)$  such that for every $x\in D_r$ the sets spt$(w_j(x))$ are  pairwise disjoint  for different indices $j$. If $F$ is lower semicontinuous,  for every $x\in D_r $ where $w$ is differentiable it holds $F(w(x),\nabla w(x)) \ge  F(a,\nabla w(x))$,
from which we can conclude that the quasiconvexity  \eqref{eq:quasiconvexity} implies  Property \ref{hp:A}.  If $F$ is continuous, $F(w(x),\nabla w(x)) = F(a,\nabla w(x))$ and the two properties are equivalent.
\end{proof}

\subsection{Proof of the main result}
We are now ready to prove Theorem \ref{thm:qcvx_non_poly}.

\begin{proof}[Proof of Theorem~\ref{thm:qcvx_non_poly}] 
We show that the $Q$-integrand $A$   defined in \eqref{eq:quasi_envel} is not polyconvex. 
    Recall that 
    the simple  2-vectors $v_i=c_i\bigwedge M(X_i)$, $i=1,2,3 $ introduced in \eqref{eq:def_v}, where $c_1=c_2= \delta^{2}(4-\e^4)/4 $ and $c_3=\e^4\delta^2/2$ satisfy 
    $v_1+v_2+v_3    =2\delta^2 e_{12}=2\delta^2\bigwedge M(0)$, that is, $\bigwedge M(0)$ is equal to the following convex combination:
    \begin{equation}\label{convcomb}
    \bigwedge M(0)=\sum_{i=1}^3 \theta_i \bigwedge M(X_i), \qquad \mbox{where } \quad \theta_i:=c_i/2\delta^2, \quad \sum_{i=1}^3 \theta_i=1. 
    \end{equation}

    We fix $a\in(\R^2)^Q$  and   $f\in \mathcal T(a,0,D)$ with Lipschitz constant $L$. We recall the notation $T_0,\textbf{T}^H, \textbf{T}^M,\textbf{T}^V$ introduced in Section~\ref{sec:nonapprox}, where we again drop the dependence of $\textbf{T}$ from $f$. Recall that each point $p\in  T_0$ is an integer density point for $\textbf{T}$ admitting a tangent plane.
    We also recall that 
    $w_1,w_2$ are $\e$-horizontal while $w_3$ is $\e$-vertical. Moreover, the 2-vectors $v_i$  are such that $R(P_{v_i})=P_{w_i}$ for $i=1,2,3$, where  $R=diag(1,1,\e,\e)$.

    By the definition of $\psi$  we can   estimate  
    \begin{equation}\label{eq:bha}
     \int_D \bar \psi(\nabla f)  \ge    \M(\textbf{T}\llcorner \{p\in T_0 : T_p T_0 \neq P_{v_i} \text{ for }i=1,2,3\}).        
    \end{equation}
    Indeed, consider a point $p\in T_0$ of density $\theta(p)$. We denote $p=(x,y)\in\R^2\times \R^2$. If $\theta(p)<Q$, then 
    $$f(x)=\theta(p)[y]+z$$ 
    with $y\in\R^2$,  $z\in\mathcal A_{Q-\theta(p)}(\R^2)$ and $d(y,$spt$ z)>0$.  In particular, choosing a ball $B$ centered at $x$ with a small enough radius, we can ensure  
    \[
    d(y,\mathrm{spt}\,z)\ge 3(Q-1)L\,\mathrm{diam}(B)>0,
    \]
    so that by \cite[Proposition~1.6]{DeLSpa11} the function $f$ admits in $B$ a decomposition into two $L$-Lipschitz functions $f_1\in Lip(B,\mathcal A_{\theta(p)})$, $f_2\in Lip(B,\mathcal A_{Q-\theta(p)})$ with disjoint graphs. Since there exists the tangent plane to $T_0$ at $p$, it also holds $\nabla f_1(x)=\theta(p) [X]$ for some $X\in\R^{2\times 2}$.    
    If at such $p$  the tangent plane to ${T}_0$ is   different from  $  P_{v_1},P_{v_2},P_{v_3}$, then $\bar        \psi(\nabla f_1(x))=\theta(p)\|\bigwedge M(X)\|$. The same conclusion is straightforward if $\theta(p)=Q$, as no decomposition is required. This implies \eqref{eq:bha} since   $\|\bigwedge M(X)\|\ge 1$ for every $X\in\R^{2\times 2}.$

    We now claim that  a quantitative  lower bound  can be established for   the right-hand side of  \eqref{eq:bha}.   We denote $\textbf{G}:=R_\sharp \textbf{T}$, which is still a multivalued graph, see \eqref{eq:gauss_graph}, and $G_0:=R(T_0)$. We also denote $\textbf{G}^{NS}:=\textbf{G}\llcorner \{p\in G_0 : T_p G_0 \neq P_{w_i} \text{ for }i=1,2,3\}$, and $\textbf{G}^{i}=\textbf{G}\llcorner \{p\in G_0 : T_p G_0  = P_{w_i}\}$ for $i=1,2,3$, which are well-defined as restrictions of a rectifiable current on measurable subsets of the support. Bounding the tangential Jacobian from above by $1$, by the area formula
    \[
     \M(\textbf{G}^{NS})\le \M(\textbf{T}\llcorner \{p\in T_0 : T_p T_0 \neq P_{v_i} \text{ for }i=1,2,3\})  .
    \]
    Moreover,  since $\textbf{G}=\textbf{G}^{NS}+\textbf{G}^1+\textbf{G}^2+\textbf{G}^3$, by the linearity of the pushforward  it holds $\gamma_{\textbf{G}} =  \gamma_{\textbf{G}^{NS}}+\sum_{i=1}^3\gamma_{\textbf{G}^i}.$ In particular,  since    $e_{12}=\tfrac1{Q}\int {\omega}\,d\gamma_{\textbf{G}}(P)$, we deduce that 
    \[
    \sum_{i=1,2}\int {\omega}\cdot e_{34}\,d\gamma_{\textbf{G}^i}(P) = \int {\omega}\cdot e_{43}\,d\gamma_{\textbf{G}^{NS}}(P) + \int {\omega}\cdot e_{43}\,d\gamma_{\textbf{G}^3}(P).
    \]
    Since the pushforward of a measure is mass-preserving, we have $|\gamma_{\textbf{G}^{NS}}|(\widetilde{Gr}(2,4))= \M(\textbf{G}^{NS})$ and $|\gamma_{\textbf{G}^{i}}|(\widetilde{Gr}(2,4))= \M(\textbf{G}^{i})$. Recalling the definition \eqref{eq:def_w}, we have $w_1\cdot e_{34}=w_2\cdot e_{34}=\frac {\e^2}{4}$ and $w_3\cdot e_{43}=\frac {\e^2}{2}$, so that the previous equality reads
    \[
    \frac {\e^2}{4\|w_1\|}\sum_{i=1,2} \M(\textbf{G}^i)=  \frac{\e^2}{2\|w_3\|} \M(\textbf{G}^3)  + \int {\omega}\cdot e_{43}\,d\gamma_{\textbf{G}^{NS}}(P).
    \]
    By $|{\omega}\cdot e_{43}|\leq 1$ for $\gamma_{\textbf{G}^{NS}}$-a.e. $P$,  we conclude 
    \begin{equation}\label{eq:est_bar0}
        0\le  \frac{2\|w_1\|}{\|w_3\|}\M(\textbf{G}^3) - \sum_{i=1,2} \M(\textbf{G}^i) + \frac{4\|w_1\|}{\e^2}\M(\textbf{G}^{NS}).
    \end{equation}    
    On the other hand,      since projecting on the $h$-plane the multivalued graph $\textbf{G}$ does not produce cancellations, it holds 
    \begin{equation}\label{eq:est_bar1}
            Q\le \M(\textbf{G})=\sum_{i=1,2,3} \M(\textbf{G}^i)  + \M(\textbf{G}^{NS}).
    \end{equation}
    Summing left-hand sides and right-hand sides of \eqref{eq:est_bar0} and \eqref{eq:est_bar1},  we get 
    $$    \Big( 1+ \frac{2\|w_1\|}{\|w_3\|} \Big)\M(\textbf{G}^3) + \Big(1+\frac{4\|w_1\|}{\e^2}\Big) \M(\textbf{G}^{NS})\ge Q.
    $$
    Since $w_3$ is $\e$-vertical it holds $\M(\textbf{G}^3)\le\M( \textbf{G}^V)$, so that   using Proposition~\ref{prop:lower_bd} and Remark \ref{rmk:general_bdry} we arrive at 
    $$    
      Q\le \frac 1C \Big( 1+ \frac{2\|w_1\|}{\|w_3\|} \Big)\M(\textbf{G}^M) + \Big(1+\frac{4\|w_1\|}{\e^2}\Big)\M(\textbf{G}^{NS})
    $$
    Finally, since   the planes associated to $w_1,w_2,w_3$ are either $\e$-horizontal or $\e$-vertical, it holds $\M(\textbf{G}^M)\le \M(\textbf{G}^{NS})$.
    All in all, we have concluded a quantitative lower bound for \eqref{eq:bha}, which concludes the proof of the claim.

    Thanks to the claim, by the very definition of $A$  we  conclude that  $A(a,0)>0$.
    On the other hand, denoting $\tilde X_i=(X_i\,,\dots, X_i)\in (\R^{2\times 2})^Q$ for $i=1,2,3$, we see that 
    $A(a, \tilde X_i)= 0$, as follows testing the definition \eqref{eq:quasi_envel} with $f=\sum_{j=1}^Q [a_j+X_i\,\cdot]$. This implies that the integrand $A$ is not polyconvex, as otherwise by Remark \ref{remarkconv}, there would exist a convex function $\widetilde g_a$ defined on $(( \bigwedge^m \R^n)^+)^Q$ such that 
\[
A(a,X_1,\ldots,X_Q)
= \widetilde{g}_a(\bigwedge M(X_1),\dots, \bigwedge M(X_Q)),
\]
which, by \eqref{convcomb},  would imply
$$A(a,0 )\leq \sum_{i=1}^3 \theta_i  A(a, \tilde X_i)=0,$$
which in turn would contradict $A(a,0)>0$. This concludes the proof.
    \end{proof}

    \begin{remark}
        Referring to the previous proof, we  note that, even if the integrand $A$ is degenerate at every  point $(a,\tilde X_i)\in\A$ with $\tilde X_i=(X_i\,,\dots,X_i)\in (\R^{2\times 2})^Q$ for $i=1,2,3$, we can always modify it to produce a nondegenerate $Q$-integrand $\tilde A $, which is not polyconvex but  whose associated energy is lower semicontinuous. Indeed, for $(a,X)\in \A(\R^2 \times \R^{2\times 2}) $  we could set 
        $\tilde A(a,X):= A(a,X) +\delta \, \|X\|$. This is a nondegenerate  $Q$-integrand (since $\|X\|=\G(\sum_{j=1}^Q[X_j],Q[0])$) which is not polyconvex for $\delta$ small enough, as follows from the fact that  
        $\sum_{i=1}^3  \theta_i A(a,\tilde X_i)<A(a,0)$. On the other hand, it is immediate to check that $\tilde A$ satisfies Property~\ref{hp:B} thanks to the convexity of the norm.
    \end{remark}

    \subsection{Concluding remarks}
\label{sec:final_rmks}
In this final subsection we provide some observations and  questions that arose during the preparation of the work.

\begin{remark}\label{rmk:anisotropic}
    The envelope \eqref{eq:quasi_envel} provides a nonhomogeneous $Q$-integrand. It is not clear if it is possible to construct  an integrand which depends on the gradient variable only. For all $X\in(\R^{2\times 2})^Q$ one could try to define the space of test functions as $\mathcal T(X,D)=\bigcup_{a\in (\R^2)^Q}\mathcal T(a,X,D) $, but for such a choice it is not clear whether the quasiconvexity of the envelope \eqref{eq:quasi_envel} holds. To be precise,   Step~2 in the proof of Theorem~\ref{prop:quasi_envel} fails, since it is not clear that $\varphi^\ell \llcorner\bd D^\ell = g \llcorner\bd D^\ell$ and   that the global competitor $\varphi$ can be defined.  
    
    Even more interesting would be to strengthen Theorem~\ref{thm:qcvx_non_poly} to define a $1$-integrand $A_Q$ which is not polyconvex and such that the associated $Q$-integrand $\bar A_Q$ is quasiconvex. 
    If this were the case, Theorem~\ref{thm:qcvx_non_poly} would also answer the question posed in \cite[Remark 1.14]{DeRLeiYou}.
    One could imagine to define $A_Q$ as follows. For every $X\in\R^{2\times 2}$
    \begin{equation}\label{eq:def_Atilde}
         A_Q(X) :=  \inf_{q\le Q} \left\lbrace \frac{1}{q}\fint_B \bar \psi(\nabla f) :  f\in \mathcal T (q[X],B)  \right\rbrace.
    \end{equation} 
    It is not difficult to see, following the proof of Theorem~\ref{thm:qcvx_non_poly}, that for every $Q\in\N$ the $1$-integrand $A_Q$ is not polyconvex. 
    Yet,  it is not clear whether the energy functional associated to $\bar A_Q$ is weakly lower semicontinuous\footnote{Note that, if that were true, a standard regularization by convolution would allow to find a smooth $1$-integrand $\psi_Q$, close to $A_Q$ and  not  polyconvex, but with $\bar \psi_Q$ quasiconvex (see \cite[p. 479]{BurIva}).}, that is the validity of   Theorem~\ref{prop:quasi_envel}. 
    The issue lies again in Step~2 in the proof of Theorem~\ref{prop:quasi_envel}, as we now sketch.
    Let $X\in \R^{2\times 2},Q\in \N$ be   such that  
    \begin{equation}\label{def:qcvx-LIN}
        \inf \left\lbrace \fint_B \bar \psi(\nabla f) :  f\in \mathcal T (Q[X],B)  \right\rbrace <  \inf \left\lbrace \fint_B \bar \psi(\nabla f) :  f\in \mathcal T ((Q-1)[X],B)  \right\rbrace,
    \end{equation}
    i.e., $A_Q(X)<A_{Q-1}(X)$.
    If such $X,Q$ did not exist, by \cite[Theorem 1.8]{DeRLeiYou} the function $\psi$ would be polyconvex.  
    Now, consider a piecewise affine  function $g$, and suppose that at a Lebesgue point $x$ for $\nabla g$ it holds   $\nabla  g= q[X] +  \sum_{j=1}^J q_j[X_j]$, with $q<Q$ and $X_j\neq X$.   We need to approximate the energy $\bar  A_Q(\nabla g(x))=qA_Q(X)+\sum_j q_j A_Q(X_j)$, as we did in \eqref{eq:est_A_2}. 
    Since $A_Q(X)<A_{Q-1}(X)$, a $Q$-valued function $f$ is  needed to  approximate the value  $A_Q(X)$ using \eqref{eq:def_Atilde}. As there is left the term $\sum_j q_j A_Q(X_j)$ to approximate, we need to add to $f$ at least one more value. Therefore, in order to approximate $\bar  A_Q(\nabla g(x))$, we need a function $\varphi$ taking at least $Q+1$ values. However, since in \eqref{eq:def_Atilde} we are only allowed to use $q$-valued maps with $q\leq Q$, $\varphi$ is not an admissible competitor. Hence we cannot conclude the last inequality of the proof of Theorem~\ref{prop:quasi_envel}. 
This argument also highlights that we cannot exclude that for every $1$-integrand $\psi$ the  envelope 
    $$\sup\{\xi\leq \psi:  \mbox{the associated $Q$-integrand $\bar \xi$ is quasiconvex}\}$$
    coincides with the polyconvex envelope of $\psi$. 
\end{remark}

\begin{remark}
    We lastly present a comment on   Property~\ref{hp:A}, which replaces the quasiconvexity property of Definition~\ref{def:q-qcvx} introduced in \cite{DeLFocSpa}. 
    As observed in Remark~\ref{rmk:A_not_C0}, we cannot expect in general the continuity of $A$. Hence the weak lower semicontinuity of the associated energy functional is not equivalent to quasiconvexity of $A$.   
    As shown in Proposition~\ref{prop:iff}, under the assumption \eqref{eq:dep_F}, 
    the weak lower semicontinuity of the energy turns out to be equivalent to Property \ref{hp:A}, which thus seems to be the natural ellipticity condition for the existence of energy minimizers in this context. Moreover, when the integrand $A$ is in addition continuous,  Property~\ref{hp:A} is equivalent to the quasiconvexity of $A$ by Lemma~\ref{lem:iff}, and our results match with those of \cite{DeLFocSpa}.
    \end{remark}

\section{An almost-polyhedral approximation result}\label{sec:piecewise}
The aim of this last section is to establish the following approximation result, which we used in Section~\ref{sec:quasi/poly}. Essentially, we show that the energy associated to the integrand $A$ defined in \eqref{eq:quasi_envel} for any Lipschitz multivalued function $f$ can be recovered through an approximating sequence of uniformly Lipschitz functions that are piecewise affine outside a set of small measure. Actually, the same proof works for every  integrand  $F$ with bounded growth   satisfying the assumptions of Proposition~\ref{prop:iff}.

\begin{lemma}\label{lem:polyapprox}
    Let $(a,X)\in \A(\R^2\times \R^{2\times 2})$.  Consider $f=\sum_{j=1}^J f_j\in\mathcal T(a,X,D_r)$ $L$-Lipschitz, for $r\in(0,1)$ such that the sets spt$(f_j(x))$ are pairwise disjoint for different indices $j$. Then    
    there exists a sequence of $L'$-Lipschitz functions   $\{g^k\}_{k\in\N}\subseteq \mathcal T(a,X,D_r)$, with $L'=L'(L,Q)>0$,  with the following properties.  For every $k\in\N$ there exists $L_k\in \N$ and pairwise disjoint cubes  $\{D^k_\ell\}_{\ell=1}^{L_k}\subseteq D_r$  such that
    \begin{equation}\label{eq:prop_appprox2}
    \begin{split}
        & |D_r\setminus \bigcup_{\ell=1}^{L_k} D^k_\ell|\le \frac 3k , \\
        & g^k\llcorner D^k_\ell =T_z f\llcorner D^k_\ell \; \mbox{ for some point $z\in D^k_\ell$ of differentiability for $f$, \quad   for every  }\ell=1,\dots,L_k.
    \end{split}
\end{equation}
    Moreover, it holds  
    \begin{equation}\label{eq:prop_appprox}
         \lim_{k\to +\infty} \int_{D_r} A(g^k,\nabla g^k) =\int_{D_r}  A(f,\nabla f). 
    \end{equation}
\end{lemma}
In order to prove Lemma \ref{lem:polyapprox}, we first introduce the following notation. Consider an open bounded set $\Omega\subseteq \R^2$ .  For $r>0$, we define an $r$-cubic subdivision of  $\Omega$, if it exists,
any finite union $U$ of disjoint open cubes contained in $ \Omega$ whose side lengths are integer multiples of $r$, and such that $|\Omega\setminus U|=0$. We can then show the following intermediate result,  which is reminiscent of \cite[Lemma  4.2, 4.3]{DanPra}.

\begin{lemma}\label{lem:Linf_appr}
Let $U\subseteq \R^2$ open and bounded, and  $f:U\to \A(\R^2)$ be $L$-Lipschitz.  For every $\delta \in(0,1)$ there exists $r\in (0,\delta)$ and an $r$-cubic subdivision $\{D(z_\ell,r)\}_{\ell=1,\dots,N_\delta }$ of a set  $D^\delta\subseteq U$ such that 
\begin{align}
        &|U\setminus D^\delta |\le \delta,  \label{eq:meas_om}\\
        &D(z_\ell,3r)\subset\joinrel \subset U \label{eq:proper_subset},\quad \text{for every } \ell=1,\dots, N_\delta
\end{align} 
and for  each $\ell=1,\dots, N_\delta $ there exists an affine function $g_\ell:\R^2\to  \A(\R^2)$ such that 
\begin{align}
&\mbox{$g_\ell=T_z f$ \; for some point $z\in D(z_\ell,r)$ of differentiability for $f$ },\label{eq:choice_z}\\[1.5ex]
&\sup_{D( z_\ell,r)}\G(f,g_\ell)\leq \delta  r,\label{eq:linfty1}\\[1.5ex]
&\Big| \{\G(\nabla f, \nabla g_\ell)>\alpha\}\cap D(z_\ell,r)\Big|\le \frac\delta\alpha  r^2, \quad \forall \alpha\in (0,+\infty).\label{eq:approx_meas}
\end{align}
 Moreover, if $\sum_{j=1}^J q_j [(a_j,X_j) ]$ is a maximal decomposition for $(f(z),\nabla f(z))$, then $f$ admits a Lipschitz decomposition in $D(z_\ell,3r)$ of the form 
\begin{equation}\label{eq:DeLSpa_split_f}
    f=\sum_{j=1}^J f_j,\quad \text{ with } f_j\in Lip(D(z_\ell,3r),\mathcal A_{q_j}) \text{ for every }j=1,\dots ,J.
\end{equation}
Also, given $F\in L^1(U)$, the points $z$ in \eqref{eq:choice_z} can be chosen so that
\begin{equation}\label{eq:energy_choice}
    F(z)\le \fint_{D(z_\ell,r)} F+\delta
    \qquad\text{for every } \ell=1,\dots,N_\delta.
\end{equation}
\end{lemma}

\begin{proof} 
Fix $\delta>0$. Choose $r_0=r_0(\delta)>0$ and a set $D_0\subset\joinrel\subset U$ such that $D_0$ is the closure of the union of pairwise disjoint translations of the square $(0,r_0)^2$. Assume in addition that each square composing $D_0$ satisfies~\eqref{eq:proper_subset} and that $|U\setminus D_0|\le \delta/2$.
Then, for every $r$ such that $r_0 \in r \N$, there exists an $r$-cubic subdivision of $D_0$ made of $I_0(r)$ cubes. We denote $\{D(z_\ell,r)\}_{\ell\in I_0(r)}$ the elements of  such a  $r$-cubic subdivision of $D_0$. We define   
$I(r)\subseteq I_0(r)
$
as the set  of indices $\ell\in I_0(r)$ such that
there exists an affine function  $g_\ell:\R^2\to \A(\R^2)$ coinciding with $T_z f$ for some $z\in D(z_\ell,r)$ and satisfying
\[
    \sup_{D(z_\ell,3r)} \G(f,g_\ell)\leq \delta  r,\quad 
    \fint_{D(z_\ell,3r)} \G(\nabla f, \nabla g_\ell)\leq \delta.
\]
We then let
\begin{equation}\label{DDelta}
  D^\delta :=\bigcup_{\ell\in I(r)} D(z_\ell,r). 
\end{equation}
As  \eqref{eq:proper_subset}, \eqref{eq:choice_z} and \eqref{eq:linfty1} hold by construction, we are left to prove \eqref{eq:meas_om} and \eqref{eq:approx_meas} for $D^\delta$, and to control the multiplicities of $f,g$. 

We start by choosing radii $r(z)$ depending on the point $z\in \mathring D_0$ in the interior of $D_0$.  First of all, for every $z\in\mathring D_0$ there exists $r(z)$ such that $D(z,4r(z))\subseteq D_0\subseteq  U$ and we can furthermore take $r(z)<\delta$. Then, since $f$ is Lipschitz, by Rademacher's Theorem for multivalued functions  \cite[Theorem ~1.13]{DeLSpa11} for a.e. $z\in  D_0$ the Taylor expansion $T_z f$ of $f$ at the point $z$ is well-defined and  satisfies 
\[
\G(f(y),T_z f(y))=o(|z-y|).
\]
In particular, for a.e. $z\in \mathring D_0$ there exists $r(z)>0$ such that $D(z,4r(z))\subseteq  D_0$ and
\begin{equation}\label{eq:tay_sup}
    \G(f(y),T_z f(y)) \le\frac{\delta}{16}  |y-z|, \quad \text{for all }y\in D(z,4r(z)).
\end{equation}
We set $r(z)=0$ for all $z$ at which $f$ is not differentiable.
Since almost every point $z\in D_0$ is  a Lebesgue point for $\nabla f$,    there exists a possibly smaller $r(z)>0$ such that  
\begin{equation}\label{eq:tay_avg}
    \fint_{D(z,\rho )} \G(\nabla f,\nabla f(z))\leq \frac \delta{16} , \qquad \forall\,0<\rho\leq 4r(z).
\end{equation}
Additionally, denoting $\sum_{j=1}^{J(z)} q_j[a_j]$ a maximal decomposition of  $f(z)$, for $r(z)$ small we can ensure that
\begin{equation}\label{eq:DeLSpa_split}
    |a_j-a_{j'} |\ge 3L(Q-1)\,\text{diam}(D(z,4r(z))),\quad \forall j\neq j'.
\end{equation}
If $F\in L^1(U)$ is given, we further require $z$ to be a Lebesgue point of $F$ with respect to cubes of side length less than $r(z)$.
Outside a null set, for every $z\in D_0$ there exists $r(z)>0$ such that for every cube $D'\subset U$ with $z\in D'$ and side length smaller than $r(z)$  it holds   $
F(z)\le \fint_{D'} F+\delta $.
We choose $r$    small enough so that the set $A(r):=\big\{ z\in \mathring  D_0 :\, r(z) \leq r\big\}$ satisfies
\begin{equation}\label{Arsmall}
|A(r)|\leq \frac \delta 2.
\end{equation}
This can be done as   $A(r)\searrow \bigcap_{r\ge 0} A(r)$  as $r\to 0$, which has zero measure.

We now prove that the radius $r$ chosen above satisfies the conclusions of the lemma.
 To this end, we first claim that 
 \begin{equation}\label{useful}
     \mbox{for each $\ell\in I_0(r)$, if $D(z_\ell,r)\not\subseteq A(r)$, then $\ell\in I(r)$.}
 \end{equation}
 Indeed, assume there exists $z\in D(z_\ell,r)\setminus A(r)$ and let $g_\ell:=T_z f.$
In particular,  $r(z)>0$, hence $f$ is differentiable at $z$.  By definition of $A(r)$ and \eqref{eq:tay_sup}  we get
\[
\begin{split}
\sup_{D(z_\ell,3r)}\G(f,g_\ell) \le   \sup_{D(z,4r)}\G(f,T_z f) \le  \delta  r.
\end{split}
\]
Similarly, using \eqref{eq:tay_avg} we get
\[
\fint_{D(z_\ell,3r)} \G(\nabla f, \nabla g_\ell)\le \frac{16}9  \fint_{D(z,4r)} \G(\nabla f, \nabla f(z))\le \delta.
\]
Therefore, $\ell\in I(r)$ and we conclude the proof of \eqref{useful}.

Property \eqref{useful} immediately implies that $I(r)\neq \emptyset$, since for $\delta<|U|$ we have $|D_0|>|A(r)|$ by \eqref{Arsmall}. 
 The validity  of \eqref{eq:meas_om} for our choice of $r$ follows from $|U\setminus D_0|\le \delta/2$ and the following estimate:  
\[
| D_0\setminus D^\delta |\overset{\eqref{DDelta}}{=} 
 \bigg|\bigcup\nolimits_{\ell\in I_0(r)\setminus I(r)} D(z_\ell,r)\bigg|
\overset{\eqref{useful}}{\leq} |A(r)|\overset{\eqref{Arsmall}}{\leq}  \frac \delta 2 .
\]
Then, using \eqref{eq:tay_avg},   for every $\ell\in I(r)$ and $\alpha>0$ we get 
\[
\Big|\{\G(\nabla f,\nabla g_\ell)>\alpha\}\cap D(z_\ell,r)\Big|\le \frac\delta  \alpha r^2,
\]
which shows \eqref{eq:approx_meas}. As noted before, the properties \eqref{eq:proper_subset}, \eqref{eq:choice_z} and \eqref{eq:linfty1}  hold by construction.

Moreover, by \eqref{eq:choice_z}, the function $g_\ell$ is affine and $(g_\ell(x),\nabla g_\ell(x))$ has  the same maximal multiplicities of $(f(z),\nabla f(z))$ at every point   $x\in \R^2$. Denoting $\sum_{j =1}^J  q_j [a_j]$ a maximal decomposition for $f(z)$, by definition of $r(z)$ the property  \eqref{eq:DeLSpa_split} is satisfied  and thus by \cite[Proposition~1.6]{DeLSpa11} the function $f$ admits a  decomposition  in $D(z,4r(z))$ as 
$f= \sum_{j=1}^J f_j ,$ 
with   $f_j$ $L-$Lipschitz  taking value in $\mathcal A_{q_j}$ and with disjoint supports. In particular, since $D(z_\ell,3r)\subseteq D(z,4r(z))$, the same decomposition holds in $D(z_\ell,3r)$, showing that \eqref{eq:DeLSpa_split_f} holds.  
Lastly, since $z\in D(z_\ell,r)$ and $r<r(z)$, by the choice of $r(z)$ it holds $F(z)\le \fint_{D(z_\ell,r)}F+\delta$ and we conclude \eqref{eq:energy_choice} and thus the proof.

\end{proof}

We can now conclude the proof of Lemma~\ref{lem:polyapprox}. 

\begin{proof}[Proof of Lemma~\ref{lem:polyapprox}.]
Let $(a,X)\in\A(\R^2\times \R^{2\times 2})$ and fix an  $L$-Lipschitz function     $f\in \mathcal T(a,X,D_r)$ with Lipschitz decomposition $\sum_{j=1}^J f_j$. Let  $\sum_{j=1}^Jq_j [(a_j,X_j)]$ be a maximal decomposition for $(a,X)$. Note that by the assumption on the supports of $f_j$, $a_j\neq a_{i}$ for $j\neq i$. We divide the proof in two steps.

\noindent\textit{Step 1:} We modify the functions   produced by Lemma~\ref{lem:Linf_appr} in order to match the boundary data.  For each $k\in\N$, let  $\{g^k_\ell\}_{\ell=1,\dots, N_k}$ denote the affine functions produced by Lemma \ref{lem:Linf_appr} for the choice $f,\delta=1/k$, with each $g^k_\ell$ defined on an element of an $r_k$-cubic subdivision $\{ D(z_\ell^k,r_k) \}_{\ell=1,\dots,N_k}$ of a set  $D^k\subseteq D_r$, with $r_k\le 1/k$.   
Since $\min_{i\neq j}\min_{x\in D_r}\mbox{dist}({\mbox{spt}(f_i(x))},{\mbox{spt}(f_{j}(x))})>0$, it holds  $T_z f=\sum_{j=1}^J T_zf_j$ at every differentiability point $z\in D_r$ for $f$.  Thus, by \eqref{eq:choice_z} the affine function $g^k_\ell$ admits a decomposition into $J$ affine functions $g^k_{\ell,j}=T_z f_j\in Lip(\R^2,\mathcal A_{q_j})$. In view of \eqref{eq:linfty1},  for $k$ large enough and for every $x\in D(z^k_\ell,r_k)$, spt$(g^k_{\ell,j}(x))$ are disjoint for different indices $j$. In particular, \eqref{eq:linfty1} implies that for $k$ large enough it holds 
\begin{equation}\label{eq:linfty11}
    \sup_{x\in D(z^k_\ell,r_k)}\G^2(f,g^k_\ell)=\sum_{j=1}^J \sup_{x\in D(z^k_\ell,r_k)} \G^2(f_j,g^k_{\ell,j})\le  (r_k/k)^2.
\end{equation}
For every $j=1,\dots,J$ we define $h^k_{\ell,j}$ as the function  interpolating between $f_j$ on $\bd D(z^k_\ell,r_k )$ and $g^k_{\ell,j}$ on $D(z^k_\ell,(1-1/k)r_k )$ constructed by Lemma~\ref{lem:interp} for the choice $\sigma=\tfrac1{k-1},$ $r=(1-\tfrac1k)r_k$.
In order to estimate the Lipschitz constant of $h^k_{\ell,j}$,   we first note that for every $x\in\bd D_{r_k}$ and thanks to \eqref{eq:linfty11} it holds 
\begin{multline*}      
\G\Big(f_j(x+z^k_\ell),g^k_{\ell,j}((1-\tfrac1k)x+z^k_\ell)\Big)\\
\le \G\Big(f_j(x+z^k_\ell),f_j((1-\tfrac1k)x+z^k_\ell)\Big)+\G\Big(g^k_{\ell,j}((1-\tfrac1k)x+z^k_\ell),f_j((1-\tfrac1k)x+z^k_\ell)\Big)
\le L\,\frac {r_k}k + \frac {r_k}k.
\end{multline*}
Thus, by   \eqref{eq:Lip_h} we conclude that  $h^k_{\ell,j}$  is $L'-$Lipschitz, for a constant $L'$ not depending on $k.$

We define the following maps on $D_r$: 
\[
 g^k:=\sum_{j=1}^J \left[ f_j\chi_{D_r\setminus D^k}+\sum_{\ell=1}^{N_k}\Big( h^k_{\ell,j} \chi_{D(z^k_\ell,r_k)\setminus D(z^k_\ell,(1-\frac1k)r_k)}  + g^k_{\ell,j}\chi_{D(z^k_\ell,(1-\frac 1k)r_k)}\Big)\right].
\]  
By definition  the function $g^k$  
coincides with $g^k_\ell$ on  $D(z^k_\ell,(1-\frac1k)r_k )$ and it  has the same trace of $f$ on $\bd D(z^k_\ell,r_k )$, thus it is $L'$-Lipschitz in $D_r$. Moreover, 
since  $g^k$ admits a Lipschitz decomposition with the same multiplicities as the global decomposition of $f$, it holds    $g^k\in\mathcal T(a,X,D_r).$

By \eqref{eq:meas_om}, and since $N_k\le \frac {1}{r_k^2}$, it holds    
\begin{multline}\label{eq:meas_D}
    \left |D_r\setminus \bigcup_{\ell=1}^{N_k}   D(z^k_\ell,(1-\tfrac1k)r_k )\right |\le \frac1k+  \left |\bigcup_{\ell=1}^{N_k}D(z^k_\ell, r_k )\setminus D(z^k_\ell,(1-\tfrac1k)r_k )\right | \\\le\frac1k + N_kr_k^2 (1-(1-\tfrac1k)^2)\le  \frac{3}k ,
\end{multline}       
thus  the function  $g^k$ satisfies the properties listed in \eqref{eq:prop_appprox2} for $k$ large enough. 
Lastly, thanks to \eqref{eq:meas_D} and \eqref{eq:tay_sup} we have $g^k\to f$ in measure. Similarly, by \eqref{eq:approx_meas} we conclude that $\nabla g^k\to \nabla f$ in measure.  Passing to a subsequence we also get     $(g^k,\nabla g^k)\to (f,\nabla f)$ pointwise    a.e. in $D_r$.

\noindent\textit{Step 2:} We now show \eqref{eq:prop_appprox} for the functions $g^k$ built in the previous step. By the pointwise convergence a.e. and lower semicontinuity of $A$ we already have that $\liminf_{k\to +\infty} A(g^k,\nabla g^k)\ge A(f,\nabla f)$ outside a set of zero measure. In particular, Fatou's lemma implies
\[
\int_{D_r} A(f,\nabla f)\le \liminf_{k\to +\infty}\int_{D_r} A(g^k,\nabla g^k).
\]
To conclude \eqref{eq:prop_appprox} it is sufficient to show the reverse inequality. 
To do so, we note that, in the above  construction of \(g^k\), we
can apply Lemma~\ref{lem:Linf_appr} with $F=A(f,\nabla f)$, since by the bounds \eqref{eq:bound_0} it holds 
$A(f,\nabla f)\in L^\infty(D_r)\subset L^1(D_r)$. 
We fix $k\in\N$ and  $\ell\in\{1,\dots,N_k\}$.    Since $g^k_\ell=T_z f$ on $D(z^k_\ell,r_k)$ for some $z\in D(z^k_\ell,r_k) $, the pairs $(g^k_\ell(x),\nabla g^k_\ell(x))$ and $(f(z),\nabla f(z))$ 
have the same maximal multiplicities for every $x\in D(z^k_\ell,r_k)$ and moreover $\nabla g^k_\ell=\nabla f(z)$. Therefore, by Remark~\ref{rmk:dep_A} it holds $
A(g^k_\ell(x),\nabla g^k_\ell(x))
=
A(f(z),\nabla f(z))$ in $D(z^k_\ell,r_k)$. By \eqref{eq:energy_choice} we conclude
\begin{equation}
\fint_{D(z^k_\ell,r_k)}A(g^k_\ell,\nabla g^k_\ell)=A(f(z),\nabla f(z))\le   \fint_{D(z^k_\ell,r_k)} A(f,\nabla f) + \frac 1k.        
\end{equation}
Summing over $\ell$, and using the fact
that $D_r\setminus\bigcup_\ell D^k_\ell$ has vanishing measure together
with the uniform Lipschitz bounds and \eqref{eq:bound_0}, we obtain the
desired reverse inequality.
\end{proof}

\vspace{2ex}

{\bf Data Availability Statement.} No data was used for the research described in the
article.

\vspace{2ex}

{\bf Conflict of interest.} The authors have no known conflict of interest for this work.

\printbibliography
\end{document}